\documentclass[12pt]{amsart}

\pdfoutput=1 

\usepackage{xr}
\usepackage[shortlabels]{enumitem}
\usepackage{mathtools}
\usepackage{amssymb}
\usepackage{bm}
\usepackage{centernot}
\usepackage{slashed}
\usepackage{textcomp}
\usepackage{yhmath}
\usepackage{anyfontsize}
\usepackage{multicol}
\usepackage{tikz}
\usetikzlibrary{arrows}
\usepackage{setspace}
\usepackage[pdftex,bookmarks=true,hidelinks]{hyperref}
\usepackage[british]{babel}
\usepackage[nameinlink]{cleveref}

\usepackage{bookmark}

\usepackage[margin = 1.04in]{geometry}

\usepackage{microtype}

\DeclareMathOperator{\A}{A}
\DeclareMathOperator{\D}{D}
\DeclareMathOperator{\F}{F}
\DeclareMathOperator{\R}{R}

\DeclareMathOperator{\dom}{dom}
\DeclareMathOperator{\imm}{im}

\DeclareMathOperator{\pf}{\mathsf{Filt}_\mathrm {max}}

\newcommand \compo {\mathbin{;}}
\newcommand \pref {\mathbin{\sqcup}}
\newcommand \bmeet {\cdot}
\newcommand \bjoin {+}

\newcommand \rest {\mathbin{\vartriangleright}}
\newcommand \aand \wedge

\newcommand{\ppreceq}{\mathbin{\preceq\hspace{-1mm}\preceq}}

\newcommand{\meet}{\textstyle\prod}
\newcommand{\join}{\textstyle\sum}

\newcommand{\from}{\colon}

\newcommand{\algebra}[1]{\mathfrak{#1}}
\newcommand{\defn}[1]{\textbf{#1}}

\makeatletter
\newcommand\mynobreakpar{\par\nobreak\@afterheading}
\makeatother

\theoremstyle{plain}
\newtheorem{theorem}{Theorem}[section]
\newtheorem{proposition}[theorem]{Proposition}
\newtheorem{corollary}[theorem]{Corollary}
\newtheorem{lemma}[theorem]{Lemma}

\theoremstyle{definition}
\newtheorem{definition}[theorem]{Definition}

\theoremstyle{remark}
\newtheorem{remark}[theorem]{Remark}
\newtheorem{problem}[theorem]{Problem}
\newtheorem{example}[theorem]{Example}

\numberwithin{equation}{section}
\newcommand{\parrow}{\rightharpoonup}

\newcommand{\typeface}{\mathbf}
\newcommand{\hauset}{\typeface{
    HausEt}}
\newcommand{\hausp}{\typeface{
    Haus}_{\mathrm p}}

\newcommand{\ba}{\typeface{BA}}
\newcommand{\ralg}{\typeface{DRA}}
\newcommand{\subalg}{\typeface{SubtrAlg}}
\newcommand{\fralg}{\typeface{C}_\mathrm{fin}\ralg}
\newcommand{\lczdet}{\typeface{Stone}{^+}\typeface{Et}}
\newcommand{\stone}{\typeface{Stone}}
\newcommand{\CC}{\mathbb C}
\newcommand{\cA}{\algebra A} \newcommand{\cB}{\algebra B}
\newcommand{\cC}{\algebra C} 
\newcommand{\cP}{\algebra P} 
 \newcommand{\cF}{\algebra F}

\newcommand{\just}[2]{\stackrel{#1}{#2}}

\newcommand*{\textrel}[1]{\mathrel{\textnormal{#1}}}

\newcommand{\ko}[1]{{\mathsf{KO}(#1)}}
\newcommand{\filt}[1]{{\mathsf{Filt}}(#1)}
\newcommand{\dbracket}[1]{\text{\textup{\textlbrackdbl}} #1
\text{\textup{\textrbrackdbl}}}

\begin{document}
\title{Difference--restriction algebras with operators}
\author{C\'elia Borlido}  \address{Universidade de Coimbra, CMUC,
Departamento de Matem\'atica, Coimbra,
Portugal}
\email{cborlido@mat.uc.pt}
\author{Ganna Kudryavtseva}
\address{University of Ljubljana, Faculty of Mathematics and Physics / Institute of Mathematics, Physics and
Mechanics, Ljubljana, Slovenia}
\email{ganna.kudryavtseva@fmf.uni-lj.si}
\author{Brett McLean}
\address{Department of Mathematics: Analysis, Logic and Discrete Mathematics, Ghent University, Ghent, Belgium}
\email{brett.mclean@ugent.be}




\newcommand{\red}{\color{red}}
\newcommand{\blue}{\color{blue}}

\keywords{Partial function, \'etale space, Stone duality, compatible completion, operators}

\begin{abstract}
We exhibit an adjunction between a category of abstract algebras of
partial functions that we call difference--restriction algebras and
a category of Hausdorff \'etale spaces. Difference--restriction
algebras are those algebras isomorphic to a collection of partial
functions closed under relative complement and domain restriction. 
Our adjunction generalises the
adjunction between the category of generalised Boolean algebras and the category
of Hausdorff spaces. We define the finitary compatible completion of
a difference--restriction algebra and show that the monad induced
by our adjunction yields the finitary compatible completion of any
difference--restriction algebra. As a corollary, the adjunction
restricts to a duality between the finitarily compatibly complete
difference--restriction algebras and the locally compact
zero-dimensional Hausdorff \'etale spaces, generalising the duality
between generalised Boolean algebras and locally compact
zero-dimensional Hausdorff spaces. We then extend these adjunction,
duality, and completion results to difference--restriction algebras
equipped with arbitrary additional compatibility preserving
operators.
\end{abstract}

\maketitle 



{}
\section{Introduction}
Functions are fundamental and ubiquitous throughout the formal sciences. \emph{Algebras of partial functions} do for functions what Boole \cite{Boole1854} did so successfully  for propositions. 
Functions---not necessarily with coincident domains, and thus `partial'---are treated as abstract elements that may be combined algebraically using various natural operations such as composition, domain restriction, override, or update. In this way, any collection of functions closed under the chosen operations forms an algebra. This allows the full power of abstract algebra to be brought to bear on the study of  functions.


In mathematics, algebras of partial functions appear for example as inverse semi\-groups~\cite{wagnergeneralised}, pseudo\-groups~\cite{LAWSON2013117}, and skew lattices~\cite{Leech19967}; in  computer science, in theories of finite state transducers \cite{10.1145/2984450.2984453} and   computable functions \cite{JACKSON2015259}; and in logic by giving semantics for deterministic propositional dynamic logics \cite{DBLP:journals/ijac/JacksonS11}, and for separation logic \cite{disjoint}. 
Many different selections of operations have been considered, each leading to a different class/category of algebras \cite{schein, garvacii71, Tro73, schein1992difference, 1018.20057, 1182.20058, DBLP:journals/ijac/JacksonS11, hirsch, BERENDSEN2010141, JACKSON2021106532}. (See \cite[\S 3.2]{phdthesis} for a guide to this literature.) 

In the algebraic study of propositions, a further seminal development occurred with the appearance of J\'onsson and Tarski's \emph{Boolean algebras with operators} (BAOs) \cite{1951, d9b8bf93-569b-32ea-ac64-293dc8ed9a8b}. The theory of Boolean algebras equipped with arbitrary additional operators has been very useful in a number of areas, most notably (the many forms of) modal logic \cite{goldblatt}. Much of the success of BAOs  can be attributed to them being dually equivalent to a category of relational structures, the \emph{descriptive general (Kripke) frames} \cite[Chapter 5]{blackburn_rijke_venema_2001}. As a result, BAOs solve the problem of finding semantics for normal modal logics (even the Kripke-incomplete ones), provide the machinery for constructing ultrafilter extensions of Kripke frames, and so on.




Returning to algebras of partial functions, although dualities/equivalences for some of the various categories of algebras have been and continue to be discovered \cite{lawson_2010, lawson2012non, kudryavtseva2017perspective, LAWSON201677, LAWSON2013117, 2009.07895, Bauer_2013, kudryavtseva2016boolean, COCKETT2021108030, KUDRYAVTSEVA2025110313}, no general framework of the sort provided by BAOs has previously existed. This is what we supply in the present work. 


In this paper, we work with the category of algebras of partial functions obtained from the signature consisting of the standard set-theoretic \emph{difference} operation and a \emph{domain restriction} operation. We refer to the algebras in this category as \emph{difference--restriction algebras}. In \cite{diff-rest1} and \cite{diff-rest2}, the first- and third-named authors began the development of a general and modular framework for partial function dualities by investigating difference--restriction algebras, arguing that they are an appropriate base class on which to build a theory analogous to BAOs. Since the union of two partial functions is not always a function, (and this determination cannot be made solely from the inclusion/extension ordering), it is valuable to be able to determine when two elements \emph{are} compatible in this sense. Not only are difference--restriction algebras very well behaved as ordered structures, acting in many ways like Boolean algebras, but domain restriction allows compatibility to be determined algebraically. 

In \cite{diff-rest1}, the authors proved that the difference--restriction algebras form a variety, giving a finite equational axiomatisation \cite[\begin{NoHyper}Theorem~5.7\end{NoHyper}]{diff-rest1}. In \cite{diff-rest2}, a `discrete' duality  was given between the category of `compatibly complete' atomic difference--restriction algebras and a category of set quotients, mirroring\footnote{In fact, it is a generalisation, as are all of our results.} the duality between the category of complete atomic Boolean algebras (CABAs) and the category of sets. Also in \cite{diff-rest2}, the result was extended to a duality `with operators' mirroring the duality between the category of CABAs with completely additive operators and the category of Kripke frames.

The main results of the present paper are the elaboration of an
adjunction between the category of difference--restriction algebras
and the category of Hausdorff \'etale spaces (\Cref{t:adj}) and the
extension of that theorem to algebras equipped with additional `compatibility preserving' operators
(\Cref{thm:expansion}). We also show that the monad induced by the
adjunction gives a form of finitary completion of algebras
(\Cref{p:completion}/\Cref{p:completion'}) and that the adjunction
restricts to a duality between the finitarily compatibly complete
algebras and the locally compact zero-dimensional Hausdorff \'etale
spaces (\Cref{t:discrete-duality}). \Cref{c:extended-duality} extends the duality of \Cref{t:discrete-duality}  with additional compatibility preserving operators and should be viewed as a partial function counterpart to the duality between BAOs and descriptive general frames.

\subsubsection*{Structure of paper} \Cref{preliminaries} contains preliminaries, including formal definitions of the class of difference--restriction algebras. We recall the axiomatisation of this class as presented in~\cite{diff-rest1}.

In \Cref{sec:duality}, we present and prove our central result: the adjunction between the category of difference--restriction algebras and the category of Hausdorff \'etale spaces (\Cref{t:adj}).

\Cref{sec:completion} concerns completion of difference--restriction algebras. We define the notions of
finitarily compatibly complete (\Cref{def:comp}) and of a finitary
compatible completion (\Cref{def:completion}), and we prove that
finitary compatible completions are unique up to isomorphism
(\Cref{p:1_0}). We prove that the monad induced by our adjunction
yields the finitary compatible completion of a difference--restriction
algebra (\Cref{p:completion}), and we conclude that the adjunction
restricts to a duality between the finitarily compatibly complete
difference--restriction algebras and the locally compact
zero-dimensional Hausdorff \'etale spaces (\Cref{t:discrete-duality}).

\Cref{sec:operators} concerns additional operations. We define the notion of a compatibility preserving operator (\Cref{def:compatibility-preserving}) and extend the adjunction (\Cref{thm:expansion}), completion (\Cref{p:completion'}), and duality (\Cref{c:extended-duality}) results of the previous two sections to difference--restriction algebras equipped with such operators.

In \Cref{sec:subtraction}, we explain how our duality is a generalisation of the duality between generalised Boolean algebras and locally compact zero-dimensional Hausdorff spaces, by describing how our result restricts to this special case (\Cref{c:1} and
\Cref{c:2}).


\section{Algebras of functions}\label{preliminaries}

Given an algebra $\algebra{A}$, when we write $a \in \algebra{A}$ or
say that $a$ is an element of $\algebra{A}$, we mean that $a$ is an
element of the domain of $\algebra{A}$. Similarly for the notation $S
\subseteq \algebra{A}$ or saying that $S$ is a subset of
$\algebra{A}$. We follow the convention that algebras are always
nonempty. If $S$ is a subset of the domain of a map $\theta$ then
$\theta[S]$ denotes the set $\{\theta(s) \mid s \in S\}$. We use
$\join$ and $\meet$ respectively as our default notations for joins
(suprema) and meets (infima). Thus $\bjoin$ and $\bmeet$ are our
default notations for \emph{binary} joins and meets.  Given an $n$-ary
operation $\Omega$ on $\cA$ and subsets $S_1, \dots, S_n \subseteq
\cA$, we shall use $\Omega(S_1, \dots, S_n)$ to denote the set
$\{\Omega(s_1, \dots, s_n) \mid s_1 \in S_1,\dots, s_n \in S_n\}$, and
if one of the sets is a singleton $\{a\}$, we may simply write~$a$.

We begin by making precise what is meant by partial functions and
algebras of partial functions. 
\begin{definition}
Let $X$ and $Y$ be sets. A \defn{partial function} from $X$ to $Y$
is a subset $f$ of $X \times Y$ validating
\begin{equation*}
(x, y) \in f \textrel{and} (x, z) \in f \implies y = z.
\end{equation*}
If $X = Y$ then $f$ is called simply a partial function on $X$. 
For a partial function $f \subseteq X \times Y$, if $(x,y)$ belongs to
$f$ then we may write $y = f(x)$. 
Given such a partial function, its \defn{domain} is the set
\[\dom(f) \coloneqq \{x \in X \mid \exists \ y \in Y \from (x, y) \in f\}.\]
For any binary relation $R \subseteq X \times Y$, we write $R^{-1}$
for the relation $\{(y, x) \mid (x, y) \in R\}$. Notice that a partial
function $f \subseteq X \times Y$ is injective if and only if $f^{-1}$
is also a partial function. If $R \subseteq X \times Y$ is a relation
and $P \subseteq X$ is a subset, then we write $R(P)$ for the subset
$\{y \in Y \mid \exists x \in P\colon (x, y) \in R\}$. Finally, for
any binary relations $R \subseteq X \times Y$ and $S \subseteq Y
\times Z$, we denote by $S \circ R$ (or simply $SR$) the composition
of $R$ and $S$:
\[S \circ R \coloneqq \{(x, z) \in X \times Z \mid \exists y \in Y \from (x,
y) \in R \textrel{and}(y, z) \in S\}.\]
When $R$ and $S$ are partial functions, this is their usual composition.
\end{definition}


\begin{definition}
An \defn{algebra of partial functions} of the signature $\{-,
\rest\}$ is a $\{-, \rest\}$-algebra whose elements are partial
functions from some (common) set $X$ to some (common) set $Y$ and
the interpretations of the symbols are given as follows.
\begin{itemize}

\item The binary operation $-$ is \defn{relative complement}:
\[f - g \coloneqq \{(x, y) \in X \times Y \mid (x, y) \in
f\textrel{and} (x, y) \not\in g\}.\]

\item The binary operation $\rest$ is \defn{domain
restriction}.\footnote{This operation has
historically been called \emph{restrictive multiplication},
where \emph{multiplication} is the historical term for
\emph{composition}. But we do not wish to emphasise this
operation as a form of composition.} It is the
restriction of the second argument to the domain of the first;
that is:
\[ f \rest g \coloneqq \{(x, y) \in X \times Y \mid x \in \dom(f)
\textrel{and} (x, y) \in g\}\text{.}\]
\end{itemize}
\end{definition}
Note that in algebras of partial functions of the signature $\{-,
\rest\}$, the set-theoretic intersection of two elements $f$ and $g$
can be expressed as $f - (f - g)$. We use the symbol~$\cdot$ for this
derived operation.

We also observe that, without loss of generality, we may assume $X =
Y$ (a common stipulation for algebras of partial functions). Indeed,
if $\cA$ is a $\{-, \rest\}$-algebra of partial functions from $X$ to
$Y$, then it is also a $\{-, \rest\}$-algebra of partial functions
from $X \cup Y$ to $X \cup Y$. In this case, this
non-uniquely-determined single set is called `the'
\defn{base}. However, certain properties may not be preserved by
changing the base. For instance, while a partial function is injective
as a function on~$X$ if and only if it is injective as a function
on~$X'$, this is not the case for surjectivity.

\begin{definition}
A \defn{difference--restriction algebra} is an algebra $\algebra A$ of the signature $\{ -, \rest\}$ that is isomorphic to
an algebra of partial functions. An isomorphism from $\algebra A$ to
an algebra of partial functions is a \defn{representation} of
$\algebra A$.
\end{definition}

Just as for algebras of partial functions, for any $\{-, \rest\}$-algebra
$\cA$, we will consider the derived operation $\bmeet$ defined by
\begin{equation*}
a \bmeet b \coloneqq a - (a - b).
\end{equation*}
In \cite{diff-rest1} it was shown that the
class of difference--restriction algebras is axiomatised by the following set of
equations.
\begin{enumerate}[label = (Ax.\arabic*), leftmargin = *]
\item \label{schein1} $a - (b - a) = a$
\item \label{commutative} 
$a \bmeet b = b \bmeet a$
\item \label{schein3} $(a - b) - c = (a - c) - b$
\item \label{eq:8} $(a \rest c)\bmeet(b \rest c) = (a \rest b) \rest
c$
\item \label{lifting}$ (a \bmeet b) \rest a = a \bmeet b$
\end{enumerate}

\begin{theorem}[{\cite[\begin{NoHyper}Theorem~5.7\end{NoHyper}]{diff-rest1}}]
The class of difference--restriction algebras is a variety,
axiomatised by the finite set of equations \ref{schein1} --
\ref{lifting}.
\end{theorem}


Algebras over the signature $\{-\}$ satisfying axioms \ref{schein1} --
\ref{schein3} are called \defn{subtraction algebras} and it is known
that in those algebras the $\bmeet$ operation gives a semilattice
structure (which we view as a meet-semilattice), and the downsets of
the form $a^\downarrow \coloneqq\{x \mid x \leq a\}$ are Boolean
algebras~\cite{schein1992difference}. In particular, the same holds
for difference--restriction algebras.

We will call an algebra over the signature $\{\cdot, \rest\}$ that is a semilattice with respect to $\,\bmeet\,$ and also validates axioms \ref{eq:8} and \ref{lifting} a \defn{restriction semilattice}. It has long been known that in a
restriction semilattice the operation $\rest$ is associative (see, for
example, \cite{vagner1962}). Moreover, the inequality $a \rest b \leq
b$ is valid (an algebraic proof appears in \cite{diff-rest1}) and will
be often used without further mention.

It is natural to raise the following questions.

\begin{problem} Determine the structure of the free $X$-generated
difference--restriction algebra, the free $X$-generated subtraction
algebra and the free $X$-generated restriction semilattice.
\end{problem}

This is a good place to define one further term that we will use later.

\begin{definition}\label{def:finite-join-dense}
Let $\algebra{P}$ be a poset. A subset $S$ of $\algebra P$ is
\defn{finite-join dense} (in $\algebra P$) if each $p \in \algebra
P$ is the join $\join T$ of some finite subset $T$ of $S$.
\end{definition}

\section{Adjunction for difference--restriction algebras}\label{sec:duality}
In this section we exhibit a contravariant adjunction between the category of
difference--restriction algebras, $\ralg$, and a certain category $\hauset$ whose
objects are Hausdorff \'etale spaces.

We use $\twoheadrightarrow$ to indicate a (total) surjective function between sets, and we use $\hookrightarrow$ to indicate an embedding of algebras. The notation $\parrow$ indicates a partial function. 
We may at times use a bracket-free notation for applications of functors to morphisms, for example, $Fh$ in place of $F(h)$.

Before giving the claimed adjunction, we define the two categories involved.

\begin{definition}
We denote by $\ralg$ the category whose objects are 
difference--restriction algebras, and whose morphisms are homomorphisms of $\{-, \rest\}$-algebras.
\end{definition}

The following definitions are needed for the definition of $\hauset$.

\begin{definition}\label{def:local_homeomorphism}
Let $X$ and $X_0$ be topological spaces. A \defn{local homeomorphism} is a function $\pi \from X \to X_0$ with the property that each $x \in X$ has an open neighbourhood $U$ such that $\pi(U)$ is open and the restriction $\pi|_U \from U \to \pi(U)$ is a homeomorphism (using subspace topologies).
\end{definition}

Note that a local homeomorphism is necessarily both an open and continuous map.

\begin{definition}\label{def:hausdorff_etale}
By an \defn{\'etale space}, we will mean a surjective local
homeomorphism $\pi \from X \twoheadrightarrow X_0$. We say such an
\'etale space is \defn{Hausdorff} if $X$ is Hausdorff (in which case
$X_0$ is too).
\end{definition}

Note that from the fact that an \'etale space $\pi \from X \twoheadrightarrow X_0$ is continuous, open, and surjective, it follows that the topology on $X_0$ is necessarily the quotient topology.

\begin{definition}\label{def:continuous}
We will call a partial function $\varphi \from X \parrow Y$
\defn{continuous} if whenever $V \subseteq Y$ is open in $Y$ then
$\varphi^{-1}(V)$ is open in $X$. We will say that $\varphi$ is
\defn{proper} if whenever $V \subseteq Y$ is compact then
$\varphi^{-1}(V)$ is also compact.
\end{definition}

Note that in particular, the domain of a continuous partial function must be open.

\begin{definition}\label{category2}
We denote by $\hauset$ the category whose objects are Hausdorff
\'etale spaces $\pi \from X \twoheadrightarrow X_0$, and where a
morphism from $\pi \from X \twoheadrightarrow X_0$ to $\rho \from Y
\twoheadrightarrow Y_0$ is a continuous and proper partial function
$\varphi \from X \parrow Y $ satisfying the following conditions:
\begin{enumerate}[label = (Q.\arabic*)]
\item\label{item:Q1} 
$\varphi$ \defn{preserves equivalence}: if both $\varphi(x)$ and $\varphi(x')$ are defined, then
\[ \pi(x) = \pi(x') \implies \rho(\varphi(x)) = \rho(\varphi(x')).\]
In particular, $\varphi$ induces a partial function
$\widetilde\varphi \from X_0 \parrow Y_0$ given by
\[\widetilde\varphi \coloneqq \{(\pi(x), \rho(\varphi(x))) \mid x \in
\dom (\varphi)\}.\]
\item\label{item:Q2} $\varphi$ is \defn{fibrewise injective}: for
every $(x_0, y_0) \in \widetilde \varphi$, the restriction and
co-restriction of $\varphi$ induces an injective partial map
\[\varphi_{(x_0, y_0)} \from  \pi^{-1}(x_0)\parrow \rho^{-1}(y_0),\]
\item\label{item:Q3} $\varphi$ is \defn{fibrewise surjective}: for
every $(x_0, y_0) \in \widetilde \varphi$, the induced partial map
$\varphi_{(x_0, y_0)}$ is surjective (that is, the image of $\varphi_{(x_0, y_0)}$ is the whole of $\rho^{-1}(y_0)$).\footnote{Note
that in the context of partial maps, the conjunction of
`injective' and `surjective' is not `bijective' in the sense of
a one-to-one correspondence.}
\end{enumerate}
\end{definition}

\begin{remark}
The fact that $\varphi$ is continuous implies that
$\widetilde \varphi$ is continuous as well. Indeed, since $\pi$ is
surjective and $\widetilde \varphi \circ \pi =\rho \circ \varphi$,
given an open subset $V \subseteq Y_0$, we have
\[\widetilde\varphi^{-1}(V) = \pi[\pi^{-1}(\widetilde \varphi^{-1}(V))] =
\pi[(\rho \circ \varphi)^{-1}(V)],\]
which is an open subset of $X_0$ because $\rho$ and $\varphi$ are
both continuous and $\pi$ is an open map.
\end{remark}

\begin{remark}\label{r:1}
A morphism $\varphi$ of \'etale spaces is an isomorphism if, and
only if, the map $\varphi: X \parrow Y$ is a (total) homeomorphism
and
\begin{equation}
\label{eq:4}
\pi(x) = \pi(x') \iff \rho(\varphi(x)) = \rho(\varphi(x')),
\end{equation}
for all $x, x' \in X$.
\end{remark}


In what follows, we define two functors \[F \from \ralg \to
\hauset^{\operatorname{op}}\ \text{ and }\ G \from \hauset^{\operatorname{op}} \to
\ralg\text{,}\] which we then show form an adjunction
(\Cref{t:adj}).

Before defining $F$, we first recall some notation from \cite{diff-rest1}.

\begin{definition}\label{sec:equiv}
Given a restriction semilattice $\algebra A$, the relation
$\preceq_\cA$ on $\algebra A$ is defined by
\[a \preceq_\cA b \iff a \leq b \rest a\]
and is a preorder.  We denote by $\sim_\cA$ the equivalence
relation induced by $\preceq_\cA$, and for a given $a \in \cA$ we
use $[a]$ to denote the equivalence class of~$a$.
\end{definition}

In the case that $\algebra A$ is an actual $\{\cdot, \rest\}$-algebra
of partial functions, the relation $\preceq_\cA$ is the \emph{domain
inclusion} relation $f \preceq_\cA g \iff \dom(f) \subseteq
\dom(g)$.\footnote{See \cite{Schein1970} for axiomatisations for
various signatures containing domain inclusion (as a fundamental
relation).}

The next result summarises some facts about $\preceq_\cA$ that were
proved in~\cite{diff-rest1} and will be used later.

\begin{proposition}\label{p:1}
The following statements hold for a restriction semilattice~$\cA$.
\begin{enumerate}[label = (\alph*)]
\item\label{item:2} The relation $\leq$ is contained in
$\preceq_\cA$ and, if $\cA$ has a bottom element $0$, then $[0] =
\{0\}$.
\item\label{item:3} The poset $\cA/{\sim_\cA}$ is a meet-semilattice
with meet given by $[a]\wedge [b] = [a \rest b]$. Moreover, if
$\cA$ is a difference--restriction algebra then $\cA/{\sim_\cA}$
is a subtraction algebra.
\item\label{item:4} The relations $\leq$ and $\preceq_\cA$ coincide
in each downset $a^\downarrow$. Moreover, the assignment $b
\mapsto [b]$ provides an order-isomorphism between $a^\downarrow$
and $[a]^\downarrow$.
\end{enumerate}
\end{proposition}

\begin{definition}
Given an algebra $\cA$ with an underlying semilattice structure, we
denote the set of filters of the poset $(\cA, \leq)$ by $\filt
\cA$. A \defn{maximal filter} is a proper filter that is maximal with respect to inclusion amongst all proper filters. 
We denote the set of all maximal filters of $\cA$ by $\pf(\cA)$.
\end{definition}

Maximal
filters admit the following characterisation.
\begin{lemma}[{\cite[\begin{NoHyper}Corollary~5.2\end{NoHyper}]{diff-rest1}}]\label{p:2}
A filter $F \subseteq \cA$ of a difference--restriction algebra
$\cA$ is maximal if, and only if, for all $a\in F$ and $b \in \cA$,
precisely one of $a \cdot b$ and $a - b$ belongs to~$F$.
\end{lemma}

Also in~\cite{diff-rest1}, it was shown that the operations $\,\bmeet\,$ and $\rest$  of a
restriction semilattice $\cA$ may be extended to $\filt\cA$, endowing
it with a restriction semilattice structure whose underlying partial
order is $\supseteq$. In particular, we defined on $\filt\cA$ the
domain inclusion relation, that will be denoted by~$\ppreceq_\cA$, and
\Cref{p:1} applies. Explicitly, we have
\begin{equation}
F \ppreceq_\cA G \iff G \rest F \subseteq F.\label{eq:2}
\end{equation} The equivalence relation induced by the
domain inclusion relation on the filter algebra is denoted
$\approx_\cA$. The set $\pf(\cA)$ is saturated with respect to $\approx_\cA$; that is, if $\mu \in \pf(\cA)$ and $F \in \filt\cA$, then $\mu \approx_\cA F \implies F \in \pf(\cA)$. 
When restricted to $\pf(\cA)$, the relation $\approx_\cA$ is simply
given by
\begin{equation}\label{eq:filter_equivalence} \mu \approx_\cA \nu \iff
\forall a \in \mu, \forall b \in \nu,\, a \rest b \in \nu.\end{equation}
See \cite[Section~4]{diff-rest1} for more extensive analysis of
relations between filters.
We denote the equivalence class of $\mu$ by $\dbracket \mu$.

\subsection{The functor \texorpdfstring{$F \from \ralg \to
\hauset^{\operatorname{op}}$}{F}}\label{sec:F}\hfill\par
\smallskip We let $F \from \ralg \to \hauset^{\operatorname{op}}$ be
defined as follows.

\noindent\underline{Objects}: Given a difference--restriction algebra $\cA$, the Hausdorff
\'etale space $F(\cA)$ is the canonical quotient
\[\pi_\cA \from \pf(\cA) \twoheadrightarrow \pf(\cA)/{\approx_\cA},\]
with the topology on $\pf(\cA)$ generated by the sets of the
form \[\widehat a\coloneqq\{\mu \in \pf(\cA) \mid a \in \mu\},\]
where $a$ ranges over elements of $\cA$ (and $\pf(\cA)/{\approx_\cA}$ is given the quotient topology). This defines $F$ on the
objects, so we need to show that $F(\cA)$ is indeed a Hausdorff
\'etale space.

First note that since $\widehat a \cap \widehat b = \widehat{a \bmeet b}$, the sets of the form $\widehat a$ are not just a sub-basis, but a basis for the topology. Since $\pf(\cA)/{\approx_\cA}$ is given the quotient topology, $\pi_\cA$ is continuous. For any set $\widehat a$ in the basis, $\pi_\cA(\widehat a)$ is open because for each $\nu \in \pi_\cA^{-1}(\pi_\cA(\widehat a))$ there exists $b \in \nu$, and then $\nu \in \widehat{a \rest b} \subseteq  \pi_\cA^{-1}(\pi_\cA(\widehat a))$. Hence $\pi_\cA$ is an open map. Now to see that $\pi_\cA$ is a local homeomorphism, let $\mu \in \pf(\cA)$ and take any $a \in \mu$. Then $\widehat a$ is an open neighbourhood of $\mu$. Since $\pi_\cA$ is continuous and an open map, to see that $\pi_\cA$ restricts to a homeomorphism on $\widehat a$ it only remains to observe that $\pi_\cA$ is injective on $\widehat a$. This follows from the representation theorem of \cite[Section 5]{diff-rest1}.

It is clear that $\pi_\cA$ is surjective. To see that $\pi_\cA$ is Hausdorff, suppose $\mu$ and $\nu$ are distinct maximal filters. Since maximal filters are maximal, we can find both  $a \in \mu \setminus \nu$ and $b \in \nu \setminus \mu$. Then using \Cref{p:2}, we obtain $a-b \in \mu$ and $b-a \in \nu$; that is, $\mu \in \widehat{a -b}$ and $\nu \in \widehat{b-a}$. These two open neighbourhoods are disjoint, since $\widehat{a -b} \cap  \widehat{b-a} = \widehat{(a-b) \bmeet (b-a)} = \widehat 0 = \emptyset$.

\begin{remark}
If a $\{-, \rest\}$-algebra $\cB$ is a Boolean algebra, then
necessarily $\rest$ coincides with $\,\bmeet\,$ (the meet). Then by
inspecting the definition of $\approx_\cB$, we see that it is the
identity relation on $\cB$. Thus the restriction of $F$ to Boolean
algebras is simply $\pf(\_)$, as in classical Stone duality.
\end{remark}

\begin{lemma}\label{lemma:compact}
Let $\cA$ be a difference--restriction algebra. Then a subset $S$ of
the space $\pf(\cA)$ is compact if and only if $S$ is of the form
$\widehat a_1 \cup \dots \cup \widehat a_n$ for some $a_1, \dots,
a_n \in \cA$.
\end{lemma}

\begin{proof}
Given any compact basis of a space, it is straightforward to see that a subset is compact if and only if it is a finite union of basis sets. Since the sets of the form $\widehat a$ are a basis of the space $\pf(\cA)$, 
we only need to check that each $\widehat a$ is compact. This follows from the fact that the map $a^\downarrow \to \mathcal P\{\mu \in \pf(\cA) \mid a \in \mu\}$ (where $\mathcal P$ denotes powerset) with $b \mapsto \widehat b$ is a representation of the Boolean algebra $a^\downarrow$ that is isomorphic to the Stone representation (\cite[Proposition~5.4]{diff-rest1}).
\end{proof}

\noindent\underline{Morphisms}: For defining $F$ on morphisms, we let $h \from \cA \to \cB$ be a
homomorphism of difference--restriction algebras.  Then the partial function $Fh \from
\pf(\cB) \parrow \pf(\cA)$ has domain $\cB_0 = \bigcup_{a \in
\cA}\widehat{h(a)}$, and for $\xi \in \cB_0$, we set $Fh(\xi) =
h^{-1}(\xi)$. Let us prove that $Fh$ defines a morphism in
$\hauset$. Using \Cref{p:2}, it is easily seen that $Fh$ is
a well-defined function and a routine computation shows that it
satisfies $(Fh)^{-1}(\widehat a) = \widehat{h(a)}$ for all $a \in
\cA$. In particular, $Fh$ is continuous (since the $\widehat {\phantom a}$ sets form a basis) and proper (using \Cref{lemma:compact}). Moreover, it
follows from~\eqref{eq:filter_equivalence} and the fact that $h$ is a homomorphism, that $Fh$ satisfies
\ref{item:Q1}. For \ref{item:Q2}, we let $\xi_1 \approx_\cB \xi_2$
satisfy $Fh(\xi_1) = Fh(\xi_2)$, that is, $h^{-1}(\xi_1) =
h^{-1}(\xi_2)$. Choosing $a \in h^{-1}(\xi_1) = h^{-1}(\xi_2)$, we
have $h(a) \in \xi_1 \cap \xi_2$, that is, both $\xi_1$ and $\xi_2$
belong to the downset of $(\filt \cB, \supseteq)$ determined by the
principal filter $h(a)^{\uparrow}$. Since $\xi_1 \approx_\cB \xi_2$ it
follows from \Cref{p:1}\ref{item:4} that $\xi_1 = \xi_2$.
Finally, let us check that \ref{item:Q3} holds. Given $\xi \in \cB_0$
and $\mu \in \pf(\cA)$ satisfying $\mu \approx Fh(\xi)$, we let $G$ be
the filter generated by the set $\{h(a) \mid a \in \mu\}$. From the fact that $h$ is a homomorphism, it is easy to see that $\{h(a) \mid a \in \mu\}$ is a filter base; thus $G = \{h(a) \mid a \in \mu\}^\uparrow$. Using again the fact that $h$ is a homomorphism, and the fact that $\rest$ is order-preserving in both arguments, we obtain both $G \rest \xi$ and $\xi \rest G$. That is $G \approx_{\algebra B} \xi$. Since $\pf(\algebra B)$ is saturated with respect to $\approx_{\algebra B}$, the filter $G$ is also maximal.
Then $h^{-1}(G) \supseteq \mu$. Since $h^{-1}(G)$ is nonempty, it is a maximal filter, then by maximality, we obtain $h^{-1}(G) = \mu$. That is, $Fh(G) = \mu$, witnessing fibrewise surjectivity. 

Since $F$ is given by inverse image, it is clear that $F$ validates the equalities necessary to be a functor. To summarise, we have proved the following.

\begin{proposition}
There is a functor \[F \from \ralg \to \hauset^{\operatorname{op}}\]
that maps a  difference--restriction algebra $\cA$ to
the Hausdorff \'etale space \[\pi_\cA \from \pf(\cA) \twoheadrightarrow
\pf(\cA)/{\approx_\cA}.\]
\end{proposition}

\subsection{The functor \texorpdfstring{$G \from \hauset^{\operatorname{op}} \to
\ralg$}{G}}\hfill\par
\smallskip 

\begin{definition}\label{def:section}
Given an \'etale space $\pi \from X
\twoheadrightarrow X_0$ a \defn{partial section} is a continuous partial function $f \from X_0
\parrow X $ satisfying $\pi \circ f = \mathrm{id}_{\dom(f)}$.
\end{definition}

\begin{remark}\label{r:2}
We note the following facts about partial sections of an \'etale
space $\pi \from X \twoheadrightarrow X_0$.
\begin{enumerate}
\item Suppose $U \subseteq X$ is open and $\pi$ is injective on
$U$. Then by continuity of $\pi$, the set $\{(\pi(x), x) \mid x
\in U\}$ is a partial section with (open) image $U$. Conversely,
$\pi$ is always injective on the image of a partial section (in
particular when the image is open). Thus the partial sections with
open image are in bijection with open subsets of $X$ on which
$\pi$ is injective.

\item If the image $U$ of a partial section $f$ is compact, then by
continuity of $\pi$, the domain $\pi[U]$ of $f$ is also
compact. Conversely, if $\dom(f)$ is compact, then it follows from
$f$ being a continuous partial function, that the image of $f$ is
compact. Thus compact image is equivalent to compact domain, for
partial sections.
\end{enumerate}
\end{remark}
\begin{lemma}\label{l:2}
Let $\pi \from X \twoheadrightarrow X_0$ be a Hausdorff \'etale space. Then the assignments
\begin{equation}
f \mapsto \imm(f)\qquad \text{ and }\qquad U \mapsto f_U:= \{(\pi (x), x) \mid
x \in U\}\label{eq:3}
\end{equation}
define a one-to-one correspondence between the set of partial
sections $f$ of $\pi$ whose image is compact and open and the set
$\ko\pi$ of compact open subsets $U\subseteq X$ such that $\pi|_U$
is injective. Moreover, $\ko\pi$ is a difference--restriction algebra
when equipped with the operations defined by
\begin{equation}
\begin{aligned}
U -_\pi V & := \{x \in X \mid x \in U \text{ and } x \notin V\},
\\ U \rest_\pi V & := \pi^{-1}(\pi(U)) \cap V.
\end{aligned}\label{eq:dual_operations}
\end{equation}
for all $U, V \in \ko\pi$.
\end{lemma}

\begin{proof} The fact that the assignments of~\eqref{eq:3} are
well-defined and mutually inverse is a consequence of
\Cref{r:2}.

We now verify that the collection $\cA^{\mathrm {co}}_\pi$ of partial
sections of $\pi$ whose image is compact and open is closed under
the operations of relative complement and domain restriction.

Recall that for Hausdorff spaces, compact implies closed. For all
$f,g \in \cA^{\mathrm {co}}_\pi$, the partial function $f-g$ is a
restriction of $f$, so $\pi$ is injective on the image of
$f-g$. Further, $\imm(f-g) = \imm(f) - \imm(g)$. So since $\imm(f)$
and $\imm(g)$ are clopen, so is $\imm(f-g)$. Then since $\imm(f-g)$
is a closed subset of the compact set $\imm(f)$, we see that
$\imm(f-g)$ is also compact. Thus $f - g \in \cA^{\mathrm
{co}}_\pi$.

For domain restriction, note that as $\dom(f)$ is clopen and $\pi$
is continuous, $\pi^{-1}(\dom(f))$ is clopen. Thus $\imm(f \rest g)
= \imm(g) \cap \pi^{-1}(\dom(f))$ is clopen. Then $\imm(f \rest g)$
is a closed subset of the compact $\imm(g)$, so $\imm(f \rest g)$ is
also compact. Thus $\imm(f \rest g)$ is a compact open subset on
which $\pi$ is injective (as $\imm(f \rest g) \subseteq
\imm(g)$). Hence $f \rest g \in \cA^{\mathrm {co}}_\pi$.  Thus
$\cA^{\mathrm {co}}_\pi$ is a difference--restriction algebra.

Finally, the fact that the operations on $\ko\pi$ inherited from
those on $\cA_\pi^{\mathrm{co}}$ are as in~\eqref{eq:dual_operations} is a
consequence of the equalities $\imm(f-g) = \imm(f) - \imm(g)$ and
$\imm(f \rest g) = \imm(g) \cap \pi^{-1}(\dom(f))$, holding for
(all) partial sections $f, g$ of $\pi$.
\end{proof}

\begin{example}
If $X_0$ is ${\mathbb R}$ with the standard topology, then only $\emptyset$ is compact and open. Since for any element $U$ of $\ko\pi$, the set $\pi(U)$ is necessarily compact and open, $\ko\pi$ is then the singleton $\{\emptyset\}$. In general, the further away $X_0$ is from a zero-dimensional space (which is plentiful in compact and open sets, as they form a basis of the topology), the poorer the algebra $\ko\pi$ is.
\end{example}

We now define the functor $G \from \hauset^{\operatorname{op}} \to
\ralg$.

\noindent\underline{Objects}: Given a \emph{Hausdorff} \'etale space
$\pi \from X \to X_0$, we let $G(\pi) = \ko\pi$ be the
difference--restriction algebra consisting of all compact open subsets
of $X$ such that $\pi|_U$ is injective, equipped with the operations
defined in \Cref{l:2}.

\noindent\underline{Morphisms}: To define $G$ on morphisms, take a
morphism $\varphi$ from $(\pi \from X \twoheadrightarrow X_0)$ to
$(\rho \from Y \twoheadrightarrow Y_0)$ in $\hauset$ (that is, a
morphism from $\rho$ to $\pi$ in $\hauset^{\operatorname{op}}$). We
let $G\varphi \from \ko\rho \to \ko \pi$ assign each $U \in \ko \rho$
to the set $G\varphi(U) := \varphi^{-1}(U)$. Since $\varphi$ is
continuous and proper, $\varphi^{-1}(U)$ is a compact open subset of
$X$. Moreover, if $x, x' \in \varphi^{-1}(U)$ are such that $\pi(x) =
\pi(x')$ then, by~\ref{item:Q1}, we have $\rho(\varphi(x)) =
\rho(\varphi(x'))$ and, since $\rho|_U$ is injective, $\varphi(x) =
\varphi(x')$. By~\ref{item:Q2}, it follows that $x = x'$, thereby
proving that $\pi|_{\varphi^{-1}(U)}$ is injective.
Thus $G\varphi (U)$ is indeed an element of $\ko\pi$.

\begin{lemma}\label{l:1}
For every morphism $\varphi$ as above, the function $G\varphi$
defines a homomorphism of $\{-, \rest\}$-algebras.
\end{lemma}
\begin{proof}
The map $G\varphi$ preserves the relative complement operation
because this is given by set theoretical difference and taking
preimages under any function preserves this operation. To see that
$G\varphi$ also preserves domain restriction we only need to prove
that, for all $U \in \ko\rho$, the equality
\begin{equation}
\varphi^{-1}(\rho^{-1}(\rho(U))) =
\pi^{-1}(\pi(\varphi^{-1}(U)))\label{eq:7}
\end{equation}
holds. Indeed, if that is the case then, given $U, V \in \ko \rho$,
we have
\[G\varphi(U \rest_\rho V) = \varphi^{-1}(\rho^{-1}(\rho(U)) \cap V)
= \pi^{-1}(\pi(\varphi^{-1}(U))) \cap \varphi^{-1}(V) =
\varphi^{-1}(U) \rest_\pi \varphi^{-1}(V).\]
Let us then prove~\eqref{eq:7}. Take $x \in X$. Then
\begin{align*}
x \in \varphi^{-1}(\rho^{-1}(\rho(U)))
& \iff \exists y \in U \colon \rho(\varphi(x)) = \rho(y)
\\ & \just{\ref{item:Q3}}\iff \exists x' \in X \colon \pi(x') = \pi(x), \,
\rho(\varphi(x)) = \rho(\varphi(x')), \text{ and }
\varphi(x') \in U
\\ &  \just{\ref{item:Q1}}\iff \exists x' \in \varphi^{-1}(U)
\colon \pi(x) = \pi(x')
\\ & \iff x \in \pi^{-1}(\pi(\varphi^{-1}(U))).\popQED\qed
\end{align*}
\end{proof}

We observe that, under the correspondence given by \Cref{l:2}, the
morphism $G\varphi$ is defined as in~\cite[Section~3.2]{diff-rest2},
when $\pi$ and $\rho$ are seen as set quotients. Observing that, for
every \'etale space $\pi$, under the correspondence of \Cref{l:2},
$\ko\pi$ is a $\{-, \rest\}$-subalgebra of the algebra of all partial
sections of $\pi$, \Cref{l:1} can be proved using the argument
of~\cite[Lemma~3.7]{diff-rest2}.

Since $G$ on morphisms is given by inverse image, it is clear that $G$
validates the equations necessary to be a functor. We have proved the
following.

\begin{proposition}
There is a functor
\[G \from \hauset^{\operatorname{op}} \to \ralg\]
that sends a Hausdorff \'etale space $\pi \from X \twoheadrightarrow
X_0$ to the $\{-, \rest\}$-algebra $\ko\pi$ consisting of all
compact and open subsets of $X$ on which $\pi$ is injective.
\end{proposition}

We can now prove that we have the claimed adjunction.
\begin{theorem}\label{t:adj}
The functors $F \from \ralg \to \hauset^{\operatorname{op}}$ and $G
\from \hauset^{\operatorname{op}} \to \ralg$ form an adjunction. That
is
\[F \dashv G.\]
\end{theorem}

\begin{proof}
We only need to define natural transformations $\eta \from {\rm
Id}_{\ralg} \implies G \circ F$ and $\lambda \from {\rm Id}_{\hauset} \allowbreak\implies
F \circ G$ satisfying
\begin{equation}
F \eta \circ \lambda_F = {\rm Id}_F \qquad \text{and} \qquad G\lambda
\circ \eta_G = {\rm Id}_G.\label{eq:1}
\end{equation}

Let us first define $\eta$.  Let $\cA$ be a difference--restriction
algebra. We first verify that, for each $a \in \cA$, the algebra $(G
\circ F)(\cA)$ contains the subset $\widehat a = \{\mu \in \pf(\cA)
\mid a \in \mu\}$. It is clear that $\widehat a$ is an open subset
of $\pf(\cA)$ and, by \Cref{lemma:compact}, it is also
compact. Moreover, by \Cref{p:1}\ref{item:4}, $\pi_\cA|_{\widehat
a}$ is injective.  Thus we may take $\eta_\cA(a) := \widehat a$,
that is, the representation map of
\cite[\begin{NoHyper}Corollary~5.9\end{NoHyper}]{diff-rest1}, under
the correspondence~\eqref{eq:3} (in particular, $\eta_\cA$ is itself
a representation of~$\cA$).

We now check that $\eta$ is indeed a natural transformation. Let $h
\from \cA \to \algebra B$ be a homomorphism of
difference--restriction algebras. A routine unfolding of the
definitions gives
\[(G\circ F)(h)(\eta_\cA(a)) = \{\mu \mid \mu \in \pf(\algebra B)
\textrel{and} h(a) \in \mu\} = \widehat{h(a)},\] as required.

In order to define $\lambda$, we claim that given a Hausdorff \'etale space
$\pi \from X \twoheadrightarrow X_0$, for each $x \in X$ the collection
\begin{equation}\label{eq:lambda}\lambda_\pi^0(x) \coloneqq
\{U \subseteq X\mid \text{$U$ is compact and open,
$x \in U$, and $\pi$ is injective on $U$}\}\end{equation}
is either empty or is a  maximal filter of $G(\pi) = \ko\pi$. 

Clearly the sets in $\lambda_\pi^0(x)$ are elements of $G(\pi)$. The
set $\lambda_\pi^0(x)$ is upward closed because the order on
$G(\pi)$ is set inclusion and being a neighbourhood of $x$ is
preserved by supersets. Similarly the fact that $\lambda_\pi^0(x)$
is closed under binary meets follows from the fact that binary meets
in $G(\pi)$ are intersections and containing $x$ is preserved by
intersections. To see that $\lambda_\pi^0(x)$ is maximal, assuming
it is nonempty, let $U \in \ko\pi \setminus \lambda_\pi^0(x)$ and
choose any $V \in \lambda_\pi^0(x)$.  Then $x \in V -_\pi U$, hence
$V-_\pi U \in \lambda_\pi^0(x)$. So any filter of $\ko\pi$ extending
$\lambda_\pi^0(x)$ and containing $U$ must contain $(V-_\pi U)
\cdot_\pi U = \emptyset$, and hence is improper.

We thus define $\lambda_\pi(x) = \lambda_\pi^0(x)$ when
$\lambda_\pi^0(x)$ is nonempty and undefined otherwise. We need to
show that $\lambda_\pi$ satisfies the conditions to be a morphism of
$\hauset$ (\Cref{category2}). Consider one of the basic open sets
$\widehat U$, where $U \in G(\pi) = \ko\pi$. Since $U$ is compact
and open, so is $\lambda_\pi^{-1}(\widehat U) = U$. This shows that
$\lambda_\pi$ is continuous, and by \Cref{lemma:compact}, it is also
proper. And the $U$ above is indeed compact, as required. For
\ref{item:Q1}, suppose $\pi(x) = \pi(x')$, and take arbitrary
elements $U$ and $U'$ of the maximal filters $\lambda_\pi(x)$ and
$\lambda_\pi(x')$ respectively. That is, $x \in U$ and $x' \in
U'$. Then $x' \in \pi^{-1}(\pi(U)) \cap U' = U \rest_\pi U'$. By
\Cref{eq:filter_equivalence}, we conclude that $\lambda_\pi(x)
\approx_{\ko\pi} \lambda_\pi(x')$. For \ref{item:Q2}, let $x, x' \in
\dom(\lambda_\pi)$, and suppose $\pi(x) = \pi(x')$ and
$\lambda_\pi(x) = \lambda_\pi(x')$. Then take any $U \in
\lambda_\pi(x)$ and we have $x, x' \in U$. Since $\pi|_U$ is
injective, we must have $x = x'$.  For \ref{item:Q3}, let
$\lambda_\pi(x)$ be defined and suppose $\mu \in \pf(\ko\pi)$
satisfies $\mu \approx_{\ko\pi} \lambda_\pi(x)$. Then by
\eqref{eq:filter_equivalence} we see that for every $U \in \mu$, we
have $\pi(x) \in \pi(U)$. Let $x_U \in U$ be such that $\pi(x) =
\pi(x_U)$. Since $\mu$ is closed under $\cdot_\pi$ (given by
intersections), and $\pi$ is injective on every element of $\ko\pi$,
the elements $x_U, x_V$ must agree for all $U, V \in \ko\pi$. That
is, there is some $x' \in \bigcap \{U \mid U \in \mu\}$ such that
$\pi(x') = \pi(x)$. Then ($\lambda_\pi(x')$ is defined and)
$\lambda_\pi(x') \supseteq \mu$. By maximality, $\lambda_\pi(x') =
\mu$, witnessing the validity of \ref{item:Q3}.

We now check that $\lambda$ is a natural transformation. Let
$\varphi \from (\pi \from X \twoheadrightarrow X_0) \parrow (\rho
\from Y \twoheadrightarrow Y_0)$ be a morphism in $\hauset$. We
first check that $FG\varphi \circ \lambda_\pi$ and $\lambda_\rho
\circ \varphi$ have the same domain. Let $x \in X$ and suppose that
both $\lambda_\pi(x)$ and $FG\varphi\circ \lambda_\pi(x)$ are
defined. This is equivalent to the existence of sets $U \in \ko\pi$
and $V \in \ko\rho$ such that $x \in U$ and $x \in \varphi^{-1}(V)$,
which implies that $\varphi(x)$ is defined and $V \in
\lambda_\rho^0(\varphi(x))$, that is, $\lambda_\rho(\varphi(x))$ is
defined. Conversely, if $\varphi(x)$ and $\lambda_\rho(\varphi(x))$
are defined, then there exists $V \in \ko\rho$ such that $\varphi(x)
\in V$. Since we have already proved that $G$ is a well-defined
functor, we have that $U:= \varphi^{-1}(V) = G\varphi(V)$ belongs to
$\ko\pi$. Then the fact that $U = \varphi^{-1}(V)\in
\lambda_\pi^0(x)$ proves that $\lambda_\pi(x)$ and
$FG\varphi(\lambda_\pi(x))$ are both defined, as required. Finally,
a similar unfolding of the definitions shows that, for $x$ in the
domain of $FG\varphi\circ \lambda_\pi$, we have $FG\varphi\circ
\lambda_\pi (x) = \lambda_\rho \circ \varphi(x)$, thereby proving
that $\lambda$ is a natural transformation.

Finally, we must check that $\eta$ and $\lambda$ do indeed
satisfy~\eqref{eq:1}. For the left-hand side of~\eqref{eq:1}, let
$\cA$ be a difference--restriction algebra. We need to verify that
$F(\eta_\cA) \circ \lambda_{F(\cA)}$ is the identity on $\pi_\cA
\from \pf(\cA) \twoheadrightarrow \pf(\cA)/{\approx_\cA}$. So let
$\xi$ be a maximal filter of $\cA$. Then $\lambda_{F(\cA)}^0(\xi)$
is nonempty, since for every $a \in \xi$, the set $\widehat{a}$ is a
compact open containing $\xi$ on which $\pi_\cA$ is injective. Thus
$\lambda_{F(\cA)}(\xi)$ is defined. Then $F(\eta_\cA)$ maps
$\lambda_{F(\cA)}(\xi)$ to $\{a \in \cA \mid \eta_\cA(a) \in
\lambda_{F(\cA)}(\xi)\}$. Unfolding the definitions shows that
$\eta_\cA(a) \in \lambda_{F(\cA)}(\xi)$ is equivalent to $a \in
\xi$; thus $F(\eta_\cA)$ maps $\lambda_{F(\cA)}(\xi)$ to $\xi$, as
required.

For the right-hand side of~\eqref{eq:1}, let $\pi \from X
\twoheadrightarrow X_0$ be a Hausdorff \'etale space. We need to
verify that $G(\lambda_\pi) \circ \eta_{G(\pi)} $ is the identity on
$\ko\pi$. So let $U \in \ko\pi$. Then $\eta_{G(\pi)}$ maps $U$ to
$\widehat U = \{\mu \in \pf(\ko\pi) \mid U \in \mu\}$. Then
$G(\lambda_\pi)$ maps this to $\lambda_\pi^{-1}(\widehat U) = U$.
\end{proof}


\section{Finitary completion and duality}\label{sec:completion}

In this section we first define the appropriate analogues of having finite joins and the completion by finite joins for algebras with \emph{compatibility} information. Then we show that the monad induced by $F \from \ralg \dashv \hauset^{\operatorname{op}} \,\from\! G$ gives precisely the `completion' of the
algebra by finite compatible joins.

Our contravariant
adjunction is a generalisation of the one between generalised Boolean
algebras\footnote{A \emph{generalised Boolean algebra}
is a `Boolean algebra without a top', that is, a distributive
lattice with a bottom and a relative complement operation, $-$,
validating $a \wedge (b - a) = \bot$ and $a \vee (b - a) = a \vee
b$.} and Hausdorff topological spaces (see Section~\ref{sec:subtraction}). Like that adjunction, ours restricts to a duality: between the full subcategories of algebras that have finite compatible joins and the \emph{locally compact, zero-dimensional} Hausdorff \'etale spaces.

In order to define and better situate our analogues of having finite joins and completion by finite joins, in the presence of compatibility information, we first recall the notion of compatibility from \cite[Definition~4.1]{diff-rest2}.

\begin{definition}
Let $\cP$ be a poset. A binary relation $C$ on $\cP$ is a
\defn{compatibility relation} if it is reflexive, symmetric, and
downward closed in $\cP \times \cP$.  We say that two elements $a_1,
a_2 \in \cP$ are \defn{compatible} if $a_1 C a_2$.
\end{definition}

One can show that `reflexive,
symmetric, and downward closed' is an axiomatisation of
a rather general conception of compatibility in the following sense.

\begin{proposition}[{\cite[Proposition~4.2]{diff-rest2}}] Let $(P, \leq, C)$ be a poset equipped with a
binary relation $C$. Then $(P, \leq, C)$ is isomorphic to a poset
$(P', \subseteq, C')$ of partial functions ordered by inclusion and
equipped with the relation `agree on the intersection of their
domains' if and only if $C$ is reflexive, symmetric, and downward
closed.\footnote{The relation $C'$ was introduced
in~\cite{Vagner1960} for semigroups of partial transformations,
where it was called
\emph{semi-compatibility}. 
}
\end{proposition}

\begin{definition}\label{def:comp} A poset $\cP$ equipped with a
compatibility relation is said to have \defn{finite compatible
joins} provided it has joins of each finite set of
pairwise-compatible elements. When that is the case, we say that
$\cP$ is \defn{finitarily compatibly complete}.
\end{definition}

\begin{definition}\label{algebra_compatibility}When speaking about compatibility for difference--restriction algebras, we mean the relation that makes two
elements compatible precisely when
\[ a_1 \rest a_2 = a_2 \rest a_1.\]
\end{definition}

It is clear that for an actual $\{-, \rest\}$-algebra of partial functions, two elements are compatible
exactly when they agree on their shared domain. This is the easiest way to see that \Cref{algebra_compatibility} does indeed define a compatibility relation.

It is also worth noting that for an actual $\{-, \rest\}$-algebra of
partial functions, if a finite join $f_1 \bjoin \dots \bjoin f_n$
exists, it is necessarily the union $f_1 \cup \dots \cup f_n$. Indeed, if $f$ is an upper bound of $\{f_1, \dots, f_n\}$, then $f_1 \cup \dots \cup f_n = f-(f - f_1 - \dots - f_n)$.\footnote{We follow the usual convention that $-$ is left associative.}

\begin{lemma}\label{l:4}
For difference--restriction algebras, having finite compatible joins is equivalent to having joins of each pair of compatible elements.
\end{lemma}

\begin{proof}
Let $\cA$ be a difference--restriction algebra. Clearly if $\cA$ has finite compatible joins, then is has joins of every pair of compatible elements. Conversely, $\cA$ always contains $0$, the join of $\emptyset$, and if $\cA$ has joins of every compatible pair, then by induction it has joins of every nonempty pairwise-compatible set. The induction relies crucially on the fact that if $\{a_1, a_2, a_3\}$ is pairwise compatible, then $\{a_1, a_2 \bjoin a_3\}$ is pairwise compatible, which is easily checked.\footnote{Thus the lemma does not hold for posets in general. Consider, for example, the poset of all subsets of cardinality at most two (of some set with at least three elements), ordered by inclusion and with compatibility relation $X \mathrel C Y \iff |X \cup Y| \leq 2$.}
\end{proof}

Note that in the case that all pairs of elements are compatible,
having \emph{binary compatible joins} (i.e.~a join of each compatible
pair) is equivalent to having all \emph{binary joins}. Boolean
algebras provide examples of difference--restriction algebras where
every pair is compatible, if $-$ is the Boolean relative complement and we use
$\rest$ as the meet symbol. In the same way, generalised Boolean
algebras provide a more general class of examples. Thus `having binary
compatible joins' is a coherent generalisation of `having binary
joins' from situations where there is no compatibility information, to
those where there is.

Next we define what a finitary compatible completion of a difference--restriction algebra is. Finitary compatible completions will be unique up to isomorphism (\Cref{p:1_0}).

We first need the following definition.

\begin{definition}
Let $\cP$ be a poset. A subset $S$ of $\cP$ is \defn{finite-join
dense} in $\cP$ if every element of $\cP$ is the join of some
finite subset of $S$.

\end{definition}

From now on, we focus exclusively on our specific case of interest: difference--restriction algebras.

\begin{definition}\label{def:completion}
A \defn{finitary compatible completion} of a difference--restriction
algebra $\cA$ is an embedding $\iota \from \algebra A
\hookrightarrow \algebra C$ of $\{-, \rest\}$-algebras such that
$\cC$ is a finitarily compatibly complete difference--restriction
algebra and $\iota[\cA]$ is finite-join dense in~$\cC$.
\end{definition}

The following technical results will be used, in particular, for the
proof that finitary compatible completions are unique up to unique
isomorphism.

\begin{lemma}\label{l:distributive}
Let $\cA$ be a 
difference--restriction algebra and $S, T \subseteq \cA$ two finite
subsets of pairwise-compatible elements. \begin{enumerate}[(a)]
\item\label{item:aa}
The set
$S \rest T$ consists of pairwise-compatible elements and, if $\cA$
is finitarily compatibly complete, then
\[\join S \rest \join T = \join (S \rest T).\]

\item\label{item:bb} Suppose $S$ is nonempty and the join $\join S$
exists. Then for every $a \in \cA$ we have
\[\meet(a-S) = a-\join S.\]
\end{enumerate}
\end{lemma}

\begin{proof}
The lemma is essentially a special case of
\cite[Lemma~4.9]{diff-rest2} (where $S$ and $T$ are not restricted
to be finite). 
The stronger hypothesis on $\cA$ in
\cite[Lemma~4.9]{diff-rest2}---that it is compatibly complete---is
not needed for the proof when $S$ and $T$ are finite.
\end{proof}

\begin{lemma}\label{l:6}
Let $\cA$ and $\cB$ be subtraction algebras and $h: \cA
\to \cB$ be a homomorphism. Given $S \subseteq \cA$, if both
$\join_\cA S$ and $\join_\cB h[S]$ exist, then the equality
$h(\join_\cA S) = \join_\cB h[S]$ holds.
\end{lemma}
\begin{proof}
Since $h$ is monotone, it is clear that $h(\join_\cA S) \geq \join_\cB
h[S]$. Conversely, it suffices to show that $h(\join_\cA S) -
\join_\cB h[S] = 0$. For the sake of readability, we set $a =
\join_\cA S$. By
\Cref{l:distributive}\ref{item:bb}, and because $h$ is a
homomorphism, we have
\begin{align*}
h(\join_\cA S) - \join_\cB h[S]
&= h(a) - \join_\cB h[S]
\\ & = \meet_\cB \{h(a) - h(s) \mid s \in S\}
\\ & = h(\meet_\cA \{a - s \mid s \in S\})
\\ & = h(a - \join_\cA S) = h(0) = 0.\popQED\qed
\end{align*}
\end{proof}

\begin{lemma}\label{lemma:key}
Let $\iota \from \algebra A \hookrightarrow \algebra B$ and $\iota'
\from \algebra A \hookrightarrow \algebra C$ be embeddings of
difference--restriction algebras, and suppose that $\iota[\algebra
A]$ is finite-join dense in $\algebra B$, and that $\algebra C$ is
finitarily compatibly complete. Then the map $\theta \from \algebra
B \to \algebra C$ that sends $b = \join \iota(a_i)$, where $a_1,
\dots, a_n \in \cA$, to $\join \iota'(a_i)$ is a well-defined
embedding of $\{-, \rest\}$-algebras.
\end{lemma}

\begin{proof}
We will treat $\cA$ as a subalgebra of $\algebra B$ and write
$\iota'(a)$ as $a'$. First we check $\theta$ is well-defined. So
suppose $\join_1^n a_i = \join_1^m d_j$ (with the joins in $\algebra
B$ and the $a_i$'s, $d_j$'s in $\cA$). Then for each $a_i$ we have
$a_i - d_1 - \dots - d_m = 0$, and thus, as $\iota'$ is a
homomorphism, $a_i' - d_1' - \dots - d_m' = 0'$ (which is the least
element of $\algebra C$). It follows that $a_i' \leq \join
d_j'$. Since this holds for each $i$ we have $\join a_i' \leq \join
d_j'$. By symmetry, we have the reverse inequality also. Thus $\join
a_i' = \join d_j'$.

Now we show that $\theta$ preserves $-$. We first claim that in difference--restriction algebras, when $\join_1^n a_i$ and $\join_1^m b_j$ exist, the equation $\join_1^n a_i \cdot \join_1^m b_j = \join_1^n\join_1^m a_i \cdot b_j$ holds. Since finite joins are given by unions in $\{-,\rest\}$-algebras of partial functions, the equation $(a_1 \bjoin a_2) \cdot b = (a_1 \cdot b) \bjoin (a_2 \cdot b)$ is valid, and then the claim follows by two inductions. Having established the claim, it is straightforward to verify that $\theta$ preserves the $\cdot$ operation. Now let $b_1, b_2 \in \algebra B$ with $b_1 = \join_i a_i$ and $b_2 = \join_j d_j$ for $a_i$'s and $d_j$'s in $\cA$.  We then calculate 
\begin{align*}
\theta(b_1 - b_2)
&= \theta(\join_i a_i - \join_j d_j)  &&\text{by the hypotheses for $b_1$ and $b_2$} \\
&= \theta(\join_i(a_i - \join_j d_j)) &&\text{by a property of partial functions}\\
&= \theta(\join_i\meet_j(a_i - d_j)) &&\text{by \Cref{l:distributive}\ref{item:bb}}\\
&= \join_i\theta(\meet_j(a_i - d_j)) &&\text{by the definition of $\theta$, since $\meet_j(a_i - d_j) \in \cA$}\\
&= \join_i\meet_j\theta(a_i - d_j) &&\text{since $\theta$ preserves $\cdot$}\\
&=\join_i\meet_j(a'_i - d'_j) &&\text{since $a_i, b_j \in \cA$ and $\iota'$ is a homomorphism}\\
&= \join_i a'_i - \join_j d'_j &&\text{working in reverse}\\
&= \theta(b_1) - \theta(b_2)&&\text{by the definition of $\theta$.}
\end{align*}

Let us show that $\theta$ preserves $\rest$.  We now treat $\algebra
A$ as a subset of both $\algebra B$ and $\algebra C$, conflating
$\algebra A$ with its images under $\iota$ and $\iota'$. For $b_1,
b_2 \in \algebra B$ take finite subsets $A_1$ and $A_2$ of $\cA$ with $b_1 = \join_{\algebra B} A_1$ and $b_2 = \join_{\algebra B} A_2$. We
calculate
\begin{align*}
\theta(b_1) \rest \theta(b_2)
&= \join_{\algebra C} A_1 \rest \join_{\algebra C} A_2 &&\text{by the definition of $\theta$} \\
&= \join_{\algebra C}(A_1 \rest A_2) &&\text{by \Cref{l:distributive}\ref{item:aa}}\\
&=\theta(\join_{\algebra B}(A_1 \rest A_2))&&\text{by the definition of $\theta$}\\
&=\theta(\join_{\algebra B} A_1 \rest \join_{\algebra B} A_2)&&\text{by \Cref{l:distributive}\ref{item:aa}}\\ 
&= \theta(b_1 \rest b_2)&&\text{by the hypotheses for $A_1$ and $A_2$.}
\end{align*}

It remains to show that $\theta$ is injective. We claim that since
$\theta$ is a homomorphism, injectivity of $\theta$ amounts to having
$\theta^{-1}(0) = \{0\}$. Suppose $\theta^{-1}(0) = \{0\}$. Then
given $a, b \in \algebra A$ we have $\theta(a) = \theta(b) \implies
\theta(a-b) = 0 = \theta(b-a)$, and hence $a-b = 0 = b-a$. By an
elementary property of sets (in particular, partial functions), this implies $a = b$,
proving the claim.
Now to show $\theta^{-1}(0) = \{0\}$, suppose   $\theta(\join a_i) =
0$. By the definition of $\theta$, we have $\join a_i' = 0$, and thus $a_i' = 0$ for each $i$. Since $\iota'$ is injective, this gives $a_i = 0$ for each $i$. Thus $\join a_i = 0$. 
\end{proof}

\begin{remark} The embeddings of the lemma above admit the following topological interpretation via the duality we are going to establish in Proposition \ref{t:discrete-duality}.
We assume that every element of $\algebra C$ is a finite join of elements of $\iota'(\algebra A)$. Then $\algebra A$, $\algebra B$ and $\algebra C$ have the same \'etale space of maximal filters. The algebra $\algebra C$ is reconstructed from this \'etale space by means of the duality of Proposition  \ref{t:discrete-duality}  and is isomorphic to the algebra $\algebra C'$  consisting of all compact open local sections of this \'etale space, whereas $\algebra A$ and $\algebra B$ are isomorphic to its subalgebras $\algebra A'$ and $\algebra B'$, respectively, and we have $\algebra A' \subseteq \algebra B'\subseteq \algebra C'$.
\end{remark}

Now we can show that finitary compatible completions are unique up to
unique isomorphism.

\begin{proposition}\label{p:1_0}
If $\iota \from \algebra A \hookrightarrow \algebra C$ and $\iota'
\from \algebra A \hookrightarrow \algebra C'$ are finitary compatible
completions of the difference--restriction algebra $\cA$ then
there is a unique isomorphism $\theta \from \algebra C \to \algebra
C'$ satisfying the condition $\theta \circ \iota = \iota'$.
\end{proposition}

\begin{proof}
For uniqueness, suppose we have an isomorphism $\theta \from
\algebra C \to \algebra C'$ satisfying $\theta \circ \iota =
\iota'$. Then for each $c \in \cA$, as $\iota[\cA]$ is finite-join dense in $\algebra C$, we have $c = \join \iota(a_i)$ for some $a_1, \dots, a_n \in \cA$. Then we have $c \geq \iota(a_i)$ for each $a_i$, which gives $\theta(c) \geq \iota'(a_i)$ for each $a_i$, and we have $c - \iota(a_1) - \dots -\iota(a_n) = 0$, which gives $\theta(c) - \iota'(a_1) - \dots - \iota'(a_n) = 0$. By properties of partial functions, these two conclusions yield $\theta(c) = \join \iota'(a_i)$. Thus $\theta$ is uniquely determined by $\iota$ and $\iota'$.

For existence, we use \Cref{lemma:key} to verify that $\theta \from \join \iota(a_i) \mapsto \join \iota'(a_i)$ is a well-defined embedding $\algebra C \to \algebra C'$ with $\iota' = \theta \circ \iota$. By symmetry, $\theta$ has an inverse, and is thus an isomorphism.
\end{proof}

Hence if an algebra has a finitary compatible completion then it is `unique', for which reason we may refer to a finitary compatible completion as \emph{the} finitary compatible completion. We may also, as is common, refer to $\algebra C$ itself as the finitary compatible completion of $\algebra A$, when $\iota \from \algebra A \hookrightarrow \algebra C$ is a finitary compatible completion.

Next, we show how to explicitly construct the finitary compatible
completion of any difference--restriction algebra, by showing that the
monad on difference--restriction algebras induced by the adjunction of
\Cref{t:adj} gives precisely the finitary compatible completion of the
algebra. We then obtain a duality between the finitarily compatibly
complete difference--restriction algebras and the locally compact
zero-dimensional Hausdorff \'etale spaces. Recall from
Section~\ref{sec:F} that, for every difference--restriction
algebra~$\cA$, we use $\pi_{\algebra A}$ to denote the canonical
quotient $\pf(\cA) \twoheadrightarrow \pf(\cA)/{\approx_\cA}$ (see
also the definition of $\approx_\cA$ given by
\eqref{eq:filter_equivalence}).

\begin{lemma}\label{l:5}
Let $\pi: X \twoheadrightarrow X_0$ be a Hausdorff \'etale
space. Then the algebra $\ko\pi$ is finitarily compatibly complete.
\end{lemma}
\begin{proof}
By \Cref{l:4}, it suffices to show that $\ko\pi$ has joins of each
pair of compatible elements. Since the compatibility relation on an
actual $\{-, \rest\}$-algebra of partial functions is given by
``agreement on shared domain'', under the correspondence of
\Cref{l:2}, two elements $U, V \in \ko\pi$ are compatible precisely
when $\pi|_{U \cup V}$ is injective. Therefore, if $U, V \in \ko\pi$
are compatible, then $U \cup V \in \ko\pi$ and this is necessarily
the join of $U$ and $V$ in $\ko\pi$.
\end{proof}
\begin{theorem}\label{p:completion}
For every difference--restriction algebra $\cA$, the homomorphism
\begin{align*}
\eta_\cA \from \cA
& \to (G \circ F) (\cA)
\\ a
& \mapsto  \widehat a = \{\mu \mid \mu \in \pf(\cA)
\textrel{and} a \in \mu\}
\end{align*}
is the finitary compatible completion of $\cA$.
\end{theorem}

\begin{proof}
We already mentioned in \Cref{sec:duality} that $\eta_\cA$ is a
representation of $\cA$, which implies that $\eta_\cA$ is
injective. The finitary compatible completeness of $(G \circ
F)(\algebra A)= \ko{\pi_\cA}$ follows from \Cref{l:5}. The fact that
$\iota_\cA(\cA)$ is finite-join dense in $(G \circ F)(\algebra A)$
follows from \Cref{lemma:compact} and from having $\eta_\cA(a) =
\widehat a$, for all $a \in \cA$.
\end{proof}

\begin{definition}
We write $\fralg$ for the full subcategory of $\ralg$ consisting of
the difference--restriction algebras that are finitarily compatibly
complete.

A topological space $X$ is \defn{locally compact} (respectively,
\defn{zero-dimensional}) if every $x \in X$ has a neighborhood basis
consisting of compact (respectively, clopen) subsets of $X$. We
write $\lczdet$ for the full subcategory of $\hauset$ consisting of
the Hausdorff \'etale spaces $\pi \from X \twoheadrightarrow X_0$
such that $X$ is locally compact and zero-dimensional.
\end{definition}

\begin{remark}
It is well-known that, a Hausdorff topological space is locally
compact and zero-dimensional if, and only if, it has a basis of
compact open subsets. In what follows, we shall use this equivalence
freely.
\end{remark}

The notation for $\lczdet$ is intended to suggest thinking of locally
compact zero-dimensional Hausdorff spaces as `generalised Stone
spaces'.

Our next goal is to show that the categories $\fralg$ and $\lczdet$
are dually equivalent. For that, it will be useful to note the
following property of locally compact, zero-dimensional, and Hausdorff
\'etale spaces:

\begin{lemma}\label{l:3}
Let $\pi: X \twoheadrightarrow X_0$ be a Hausdorff \'etale space,
with $X$ locally compact and zero-dimensional. Then for all $x \in
X$, the set $\lambda_\pi^0(x)$ (recall~\eqref{eq:lambda}) is a base
of the neighbourhood filter of $x$, that is, for every open subset
$U \subseteq X$ such that $x \in U$, there exists $V \in
\lambda_\pi^0(x)$ satisfying $V \subseteq U$.
\end{lemma}
\begin{proof}
Let $x \in X$ and $U \subseteq X$ be an open neighborhood of
$x$. Since $\pi$ is a local homeomorphism, $x$ has an open
neighborhood $U' \subseteq X$ such that $\pi|_{U'}$ is injective
(actually, a homeomorphism) and thus, so is $\pi_{U \cap U'}$. On
the other hand, since $X$ is Hausdorff, locally compact and
zero-dimensional, $U \cap U'$ may be written as a union of
compact open subsets of $X$. In particular, there exists one such
$V$ satisfying $x \in V \subseteq U \cap U' \subseteq U$. Clearly,
we also have $V \in \lambda^0_\pi(x)$.
\end{proof}

\begin{proposition}\label{t:discrete-duality} There is a duality
between $\fralg$ and $\lczdet$.
\end{proposition}

\begin{proof} The following statements are sufficient to prove the
proposition.

\noindent\underline{Image of $F$ is in $\lczdet$}: We need to show
that an arbitrary $\pi_\cA$ is locally compact and zero
dimensional. The sets of the form $\widehat a$, for $a \in \cA$ are
by definition open and by \Cref{lemma:compact}, compact. Since
$\pi_\cA$ is Hausdorff, they are also closed; thus they are
clopen. We noted in \Cref{sec:F} that they form a basis, so
$\pi_\cA$ is indeed zero dimensional. Now let $\mu$ be a maximal
filter of $\cA$ and take any $a \in \mu$. Then $\widehat a$ is a
compact open neighbourhood of $\mu$, so $\pi_\cA$ is also locally
compact.

\noindent\underline{Image of $G$ is in $\fralg$}: This is a
consequence of \Cref{l:5}.

\noindent\underline{$\eta$ restricted to $\fralg$ is a natural
isomorphism}: We need to show that if $\cA$ is finitarily
compatibly complete, then $\eta_\cA$ is an isomorphism of
algebras. Assuming $\cA$ is finitarily compatibly complete, the
identity map on $\cA$ is a finitary compatible completion. By
\Cref{p:completion}, $\eta_\cA$ is also a finitary compatible
completion. Then by \Cref{p:1_0}, the map $\eta_\cA$ is the
composition of the identity with an isomorphism, i.e.~$\eta_\cA$ is
an isomorphism.

\noindent\underline{$\lambda$ restricted to $\lczdet$ is a natural
isomorphism}:
Let $\pi: X \twoheadrightarrow X_0$ be a Hausdorff \'etale space, with $X$
a locally compact and zero-dimensional space. We need to show that
$\lambda_\pi$ defines an isomorphism of \'etale spaces (recall
\Cref{r:1}). The fact that $\lambda_\pi$ is a total map follows
immediately from \Cref{l:3} and, since $X$ is Hausdorff, so does
injectivity of $\lambda_\pi$. In order to show that $\lambda_\pi$ is
surjective, it suffices to show that, for all $\mu \in \pf(\ko\pi)$,
the intersection $\bigcap \mu$ is nonempty. Indeed, if $x \in \bigcap
\mu$ then $\mu \subseteq \lambda_\pi(x)$ and, by maximality, the
equality follows. In turn, to see that $\bigcap \mu$ is nonempty, we
pick any $K \in \mu$ and observe that $\cF = \{V \cap K \mid V \in
\mu\}$ is a family of closed subsets of $K$ having the finite
intersection property. Since $K$ is compact, it follows that $\bigcap
\mu = \bigcap \cF \neq \emptyset$, as required. Let us now prove that
$\lambda_\pi$ is an open map. Given $V \in \ko\pi$, it is easy to
check that $\lambda_\pi(V) = \widehat V$. Now, for an arbitrary open
subset $U \subseteq X$ and $x \in U$, by \Cref{l:3}, there exists $V
\in \ko\pi$ such that $x \in V \subseteq U$. This yields
$\lambda_\pi(x) \in \widehat V = \lambda_\pi(V) \subseteq
\lambda_\pi(U)$, from where we may conclude that $\lambda_\pi(U)$ is
open. It remains to show that the equivalence
\[\pi(x) = \pi(x') \iff \lambda_\pi(x) \approx_{\ko\pi} \lambda_\pi(x')\]
holds for all $x, x' \in X$. The forward implication is a consequence
of property \ref{item:Q1} for $\lambda_\pi$. Conversely, suppose that
$\pi(x) \neq \pi(x')$. Since $X_0$ is Hausdorff, there exist open
subsets $U, U' \subseteq X_0$ such that $\pi(x) \in U$, $\pi(x') \in
U'$, and $U \cap U' = \emptyset$. By \Cref{l:3}, there exist $V, V'
\in \ko\pi$ such that $x \in V \subseteq \pi^{-1}(U)$ and $x' \in V'
\subseteq \pi^{-1}(U')$. Clearly, $V \in \lambda_\pi(x)$ and $V' \in
\lambda_\pi(x')$. Moreover, we cannot have $x' \in \pi^{-1}(\pi(V))$
or else, we would have
\[\pi(x') \in \pi(V) \cap \pi(V') \subseteq \pi(\pi^{-1}(U)) \cap
\pi(\pi^{-1}(U')) = U \cap U',\]
contradicting emptiness of the latter intersection. By
\eqref{eq:filter_equivalence}, it follows that $\lambda_\pi(x)
\not\approx_{\ko\pi} \lambda_\pi(x')$, as intended.
\end{proof}

\begin{remark}\label{r:3}
The duality between $\fralg$ and $\lczdet$ proved in
\Cref{t:discrete-duality} is precisely the fixed point duality
induced by the adjunction of \Cref{t:adj}. Indeed, if $\cA$ is a
difference--restriction algebra such that $\eta_\cA$ is an
isomorphism, then $\cA$ is isomorphic to $GF\cA$ and, in the proof
of \Cref{t:discrete-duality}, we have shown that the image of $G$
consists of finitarily compatibly complete difference--restriction
algebras. Thus, $\cA \in \fralg$. Likewise, we can argue that $\pi
\in \lczdet$ whenever $\lambda_\pi$ is an isomorphism.
\end{remark}

\begin{definition}
  Given partial functions $f, g$ from $X$ to $Y$, their
  \defn{preferential union} (also known as \defn{override}) is defined
  by
  \[f \pref g = \{(x, y) \mid (x, y) \in f  \text{ or } ((x, y) \in g
    \text{ and } x \notin \dom(f))\}.\]
\end{definition}

\begin{corollary}\label{t:duality_pref}
  There is a duality between the category $\CC$ of representable
  $\{-,\rest,\pref\}$-algebras and $\{-, \rest,
  \pref\}$-homomorphisms, and
  $\lczdet$.\footnote{In \cite{Stokes2010}, a finite equational axiomatisation was given for the representable  $\{-, \rest, \pref\}$-algebras. See~\cite{jackson2021restriction} for a list of
    signatures that are term-equivalent to $\{-, \rest, \pref\}$.}
\end{corollary}

\begin{proof}   By \Cref{t:discrete-duality}, it suffices to show that the
  categories $\fralg$ and $\CC$ are equivalent; in fact they are isomorphic. We clearly have a
  forgetful functor $U: \CC \to \fralg$. Conversely, we define $H:
  \fralg \to \CC$ as follows. \setlength{\columnsep}{1mm}
  \begin{itemize}
  \item If $(\cA, -, \rest)$ is a finitarily compatibly complete
    difference--restriction algebra, then $H(\cA) = (\cA, -, \rest,
    \pref)$, where
    \[a \pref b \coloneqq a \vee (b - (a\rest b)),\]
    for all $a, b \in \cA$.  This operation is well-defined since $a$
    and $b- (a \rest b)$ are compatible elements. In any algebra of partial functions (where joins are necessarily unions) this operation is clearly preferential union. It follows that any representation of $(\cA, -, \rest)$ is also a representation of $(\cA, -, \rest,
    \pref)$. So since $(\cA, -, \rest)$ is a representable $\{-, \rest\}$-algebra, we know that $H(\cA)$ is a representable $\{-, \rest, \pref\}$-algebra. 
  \item If $h: \cA \to \cB$ is a $\{-, \rest\}$-homomorphism,
    then we set $H(h) = h$. By \Cref{l:6}, the map $h$ preserves $\pref$ and
    thus, it is a homomorphism in $\CC$.
  \end{itemize}
  It is then clear that $H$ is a well-defined functor such that
  $U\circ H$ is the identity functor on $\fralg$. To show these
  functors determine an isomorphism, it remains to show that $H \circ U$ is the identity
  functor on~$\CC$. This is obvious at the level of
  morphisms, and also holds at the level of objects since $a \pref b = a \vee (b - (a\rest b))$ is valid in any representable $\{-, \rest, \pref\}$-algebra. 
\end{proof}

\begin{remark}
  While we have shown in the proof of \Cref{t:duality_pref} that the
  categories of finitarily compatibly complete difference--restriction
  algebras and of representable $\{-, \rest, \pref\}$-algebras are
  equivalent, it is proved in~\cite{Leech1990} that the latter is
  equivalent to the category of \emph{right-handed skew Boolean intersection
    algebras} and a duality for this category is developed
  in~\cite{skewwithintersection}. The interested reader can check that
  the duality of~\cite[Theorem~4.13]{skewwithintersection} is the same
  as ours. Other, also related, dualities for skew Boolean algebras
  can be found in~\cite{Kudryavtseva2012, kudryavtseva2016boolean}.
\end{remark}

\begin{corollary}
The category $\fralg$ is a reflective subcategory of $\ralg$, with reflection given by finitary compatible completion.
\end{corollary}

\begin{proof}
We saw in the proof of \Cref{t:discrete-duality} that the restriction of $G \circ F$ to $\fralg$ is naturally isomorphic to the identity. It is a direct consequence that $G \circ F$, viewed as a functor $\ralg \to \fralg$, is left adjoint to the inclusion $\fralg \to \ralg$. By \Cref{p:completion}, the reflector $G \circ F$ is finitary compatible completion.
\end{proof}

\begin{corollary}
The category $\lczdet$ is a reflective subcategory of $\hauset$.
\end{corollary}

\begin{proof}
We saw in the proof of \Cref{t:discrete-duality} that the restriction of $F \circ G$ to $\lczdet$ is naturally isomorphic to the identity. Thus $F \circ G$, viewed as a functor $\hauset \to \lczdet$, is left adjoint to the inclusion $\lczdet \to \hauset$. 
\end{proof}


We have now achieved all our main objectives for this section. We end by noting that there are several equivalent definitions of a finitary compatible completion possible, just as there are  several equivalent definitions of completions of Boolean algebras~\cite{completion}.

\begin{proposition}\label{equivalent} Let $\iota \from \algebra A \hookrightarrow
\algebra C$ be an embedding of difference--restriction algebras. The
following are equivalent.
\begin{enumerate}[label = (\alph*)]
\item\label{item:a} $\algebra C$ is finitarily compatibly complete, and the
image of $\algebra A$ is finite-join dense in $\algebra C$.
\item\label{item:c} $\algebra C$ is the `smallest' extension of
$\algebra A$ that is finitarily compatibly complete. That is,
$\algebra C$ is finitarily compatibly complete, and for every
other embedding $\kappa \from \cA \hookrightarrow \cB$ into a
finitarily compatibly complete difference--restriction
algebra~$\cB$, there exists an embedding $\widehat \kappa \from
\cC \hookrightarrow \cB$ making the following diagram commute.
\begin{center}
\begin{tikzpicture}[node distance = 20mm]
\node (A) at (0,0) {$\cA$}; \node[right of = A] (C) {$\cC$};
\node[below of = C, yshift = 5mm] (B) {$\cB$};

\draw[right hook->] (A) to node[above] {$\iota$} (C); \draw
[right hook->] (A) to node[below]{$\kappa$} (B); \draw[right hook->,
dashed] (C) to node[right] {$\widehat \kappa$} (B);
\end{tikzpicture}
\end{center}
\item\label{item:d} $\algebra C$ is the `largest' extension of
$\algebra A$ in which the image of $\algebra A$ is finite-join
dense. That is, $\iota[\algebra A]$ is finite-join dense in
$\algebra C$, and for every other embedding $\kappa \from \cA
\hookrightarrow \cB$ into a difference--restriction algebra~$\cB$
in which the image of $\algebra A$ is finite-join dense, there
exists an embedding $\widehat \kappa \from
\cB \hookrightarrow \cC$ making the following diagram commute.
\begin{center}
\begin{tikzpicture}[node distance = 20mm]
\node (A) at (0,0) {$\cA$}; \node[right of = A] (C) {$\cB$};
\node[below of = C, yshift = 5mm] (B) {$\cC$};

\draw[right hook->] (A) to node[above] {$\kappa$} (C); \draw
[right hook->] (A) to node[below]{$\iota$} (B); \draw[right hook->,
dashed] (C) to node[right] {$\widehat \kappa$} (B);
\end{tikzpicture}
\end{center}
\end{enumerate}
\end{proposition}
\begin{proof} 
For \ref{item:a} $\implies$ \ref{item:c}, apply \Cref{lemma:key} to
$\iota$ and $\kappa$, using join density of $\iota[\algebra A]$ in
$\algebra C$. For \ref{item:a} $\implies$ \ref{item:d}, apply
\Cref{lemma:key} to $\kappa$ and $\iota$, using finitary compatible
completeness of $\algebra C$.

For the reverse implications, let $\kappa \from \algebra A \hookrightarrow
\algebra B$ be the finitary compatible completion of $\cA$ (whose existence is
guaranteed by \Cref{p:completion}).

If \ref{item:c} holds for $\algebra C$, then it can be applied to
$\kappa \from \algebra A \hookrightarrow \algebra B$. Then
$\widehat\kappa$ is surjective. Indeed, since $\kappa[\algebra A]$
is finite-join dense in $\algebra B$, each $b \in \cB$ may be
written as a finite join $b = \join_{i = 1}^n \kappa(a_i)$, for some
$\{a_1, \dots, a_n\} \subseteq \cA$. Then since $\cC$ is finitarily
compatibly complete, the join $\join_{i = 1}^n \iota(a_i)$ exists in
$\cC$ and, by \Cref{l:6}, we have
\[b = \join_{i = 1}^n \kappa(a_i) = \widehat k(\join_{i = 1}^n
\iota(a_i)).\]
Hence the embedding $\widehat\kappa$ is in fact an isomorphism,
and so~\ref{item:a} holds.

Similarly, if~\ref{item:d} holds for $\algebra C$, then apply it
to $\kappa \from \algebra A \hookrightarrow \algebra B$ to obtain
$\widehat\kappa \from \algebra B \hookrightarrow \algebra C$. Then
as the image of $\algebra A$ is also finite-join dense in
$\algebra B$, by \Cref{lemma:key}, we have $\widehat\iota \from
\algebra C \hookrightarrow \algebra B$ commuting with the
embedding of $\algebra A$. The composition  $\widehat\kappa \circ
\widehat\iota$   is a homomorphism fixing the embedded copy of
$\algebra A$ (which is finite-join dense), and hence by \Cref{l:6},  $\widehat\kappa \circ
\widehat\iota$
is the identity on $\cC$. Similarly  $\widehat\iota \circ \widehat\kappa$ is the identity on $\cB$. So $\widehat\kappa$ is an
isomorphism and \ref{item:a} holds.
\end{proof}

\section{An adjunction for difference--restriction algebras with
operators}\label{sec:operators}

In this section we extend the adjunction, completion, and duality
results of the previous two sections to results allowing the algebras
to be equipped with arbitrary additional operators
respecting the compatibility structure. Unless specified otherwise,
let $\algebra A$ be a difference--restriction algebra.

First we introduce the class of operations we are interested in.

\begin{definition}\label{def:compatibility-preserving}
Let $\Omega$ be an $n$-ary operation on $\algebra A$. Then $\Omega$
is \defn{compatibility preserving} if whenever $a_i, a'_i$ are
compatible, for all $i$, we have that $\Omega(a_1, \dots, a_n)$ and
$\Omega(a'_1, \dots, a'_n)$ are compatible.

The operation $\Omega$ is an \defn{operator} if:
\begin{enumerate}
\item
$\Omega$ is \defn{normal}:
for any $i$ and any $a_1, \dots, a_{i-1}, a_{i+1}, \dots, a_n \in
\algebra A$, we have 
\begin{equation}\label{normal}
\Omega(a_1, \dots, a_{i-1}, \allowbreak 0,
a_{i+1}, \dots, a_n) = 0,
\end{equation}

\item
$\Omega$ is \defn{additive}:  for any $i$ and any $a_1, \dots, a_{i-1}, a_{i+1}, \dots, a_n, b, c \in
\algebra A$, if the join $b + c$ of exists, we have
\begin{equation}\label{additive}\begin{split}
&\Omega(a_1, \dots, a_{i-1}, \allowbreak b \bjoin c,
a_{i+1}, \dots, a_n) = \\&\Omega(a_1,
\dots, a_{i-1}, b, a_{i+1}, \dots, a_n) + \Omega(a_1,
\dots, a_{i-1}, c, a_{i+1}, \dots, a_n).
\end{split}
\end{equation}
\end{enumerate}
\end{definition}

This definition of an operator (making it not just a synonym for `operation') is standard in the literature, and dates back at least to J\'{o}nsson and Tarski's paper ``Boolean Algebras with Operators.~Part I'' \cite{1951}.

The algebraic categories we consider in this section take the following form, for a functional signature $\sigma$ (disjoint from $\{-, \rest\}$).

\begin{definition}
The category $\ralg(\sigma)$ has
\begin{itemize}
\item
objects: algebras of the signature $\{-, \rest\} \cup \sigma$ whose $\{-, \rest\}$-reduct is a difference--restriction algebra, and such that the symbols of $\sigma$ are interpreted as compatibility preserving operators;
\item morphisms: homomorphisms of $(\{-, \rest\} \cup
\sigma)$-algebras.
\end{itemize}
\end{definition}

For reference, we will briefly list concrete operations $\Omega$ on
partial functions according to whether or not they are compatibility
preserving operators---by which we mean the 
representable $\{-, \rest, \Omega\}$-algebras are a (necessarily full)
subcategory of $\ralg(\{\Omega\})$. 
Verifying that such inclusions hold is a simple matter, achieved by
checking the operation is a compatibility preserving operator for actual algebras of partial functions.

We can list the following compatibility preserving operators from the literature: \emph{composition} (usually denoted $\compo$), the unary $\D$ (\emph{domain}), $\R$ (\emph{range}), and $\F$ (\emph{fixset}) operations (the identity function restricted, respectively, to the domain, range, and fixed points of the argument \cite{hirsch}), the constant $1$ (\emph{identity}), and the binary $\vartriangleleft$ (\emph{range restriction}). \emph{Intersection} is of course already expressible in our base signature. 

The signature obtained by adding composition to $\{-, \rest\}$ has been studied by Schein and the representation class axiomatised under the name of \emph{difference semigroups} \cite{schein1992difference}. The signature obtained by adding composition, domain, and identity is term equivalent to $\{\compo, \bmeet, \A\}$, for which the representation class is axiomatised in \cite{DBLP:journals/ijac/JacksonS11}. Adding range to this last signature we obtain $\{\compo, \bmeet, \A, \R\}$, and the representation class is axiomatised in \cite{hirsch}. 

Operations that fail to be compatibility preserving operators usually do so because they are not even order-preserving. We can list: \emph{antidomain} (identity function restricted to complement of domain) and its range analogue \emph{antirange}, similar negative versions of $\rest$ and $\vartriangleleft$, which we might call \emph{antidomain restriction} and \emph{antirange restriction} (antidomain restriction is called \emph{minus} in \cite{BERENDSEN2010141}), \emph{override} (also known as \emph{preferential union}) and \emph{update} (see \cite{BERENDSEN2010141}), \emph{maximal iterate} (see \cite{jackson2021restriction}) and \emph{opposite} (converse restricted to points with a unique preimage \cite{finiterep}). \emph{Converse} is an operation that \emph{is} an operator but fails to be compatibility preserving. An interesting future project would be to extend the results of this section in a way that encompasses some of these operations, in particular the several that are either order-preserving or order-reversing in each coordinate. (An analogous extension of the theory of Boolean algebras with operators can be found in \cite{GEHRKE2001345}.)

Starting with an 
$n$-ary operation $\Omega$, we can define an $(n+1)$-ary relation
$R_\Omega$ on the maximal filters of $\algebra A$
by \begin{equation}\label{make_relation}R_\Omega \mu_1{\dots} \mu_{n+1}
\iff \{\Omega(a_1, \dots, a_n) \mid a_i \in \mu_i, 1\leq i \leq n\}\subseteq \mu_{n+1}.\end{equation}

Next we introduce a class of relations on Hausdorff \'etale spaces, which will turn out to be precisely the relations obtained from \emph{compatibility preserving operators} in the way just described.

\begin{definition}
Let $\pi \from X \twoheadrightarrow
X_0$ be a Hausdorff \'etale space, and let $R$ be an $(n+1)$-ary relation on $X$. 
The
\defn{compatibility relation} $C \subseteq X \times X$ is given by $x C y$
if and only if  $\pi(x) = \pi(y) \implies x = y$. Then $R$ has
the \defn{compatibility property} (with respect to~$\pi$) if given
$x_1C x'_1, \dots ,\ x_nC x'_n$ and $Rx_1{\dots} x_{n+1}$ and
$Rx'_1{\dots} x'_{n+1}$, we have $x_{n+1}C x'_{n+1}$.
\end{definition}

Let $X$ be any topological space, and let $R$ be an $(n+1)$-ary relation on $X$. Given subsets $S_1,\dots, S_n$ of $X$, we define $\Omega_R(S_1, \dots, S_n)$ by
\begin{equation}\label{make_operation}
\Omega_R(S_1, \dots, S_n) \coloneqq \{x_{n+1} \in X \mid \exists x_1\in S_1, \dots, x_n \in S_n : Rx_1{\dots} x_{n+1}\} 
\end{equation}

\begin{definition}
Let $X$ be a topological space, and let $R$ be an $(n+1)$-ary relation on $X$.

We say the relation $R$ is 
\defn{continuous} if whenever $S_1,\dots, S_n \subseteq X$ are open sets, then $\Omega_R(S_1, \dots, S_n)$ is an open set.
We say a continuous relation is   \defn{spectral} if whenever $S_1,\dots, S_n \allowbreak\subseteq X$ are compact open sets, then $\Omega_R(S_1, \dots, S_n)$ is a compact open set.

We say the relation $R$ is \defn{tight} if, for each $x_1,\dots,x_{n+1} \in X$, the condition \begin{equation}\label{eq:tight}\forall S_1,\dots,S_n \text{ compact and open }(x_1 \in S_n, \dots, x_n \in S_n \implies x_{n+1} \in \Omega_R(S_1, \dots, S_n))\end{equation} implies $R x_1, \dots, x_{n+1}$.
\end{definition}

Given an $(n+1)$-ary relation $R$ on a Hausdorff étale space that has the compatibility property and is spectral,  $\Omega_R$ defines an $n$-ary operation on the dual $\algebra A^{\mathrm {co}}_\pi$ of $\pi
\from X \twoheadrightarrow X_0$. 
Indeed, 
suppose $S_1, \dots, S_n \in \algebra A^{\mathrm {co}}_\pi$. If $x_{n+1}, x'_{n+1} \in
\Omega_R(S_1, \dots, S_n)$, then there exist $x_1\in S_1, \dots, x_n
\in S_n$ with $Rx_1{\dots} x_{n+1}$, and $x'_1\in S_1, \dots, x'_n \in
S_n$ with $Rx'_1{\dots} x'_{n+1}$. For each $i$, since $\pi$ is injective on $S_i$, we have $x_iCx'_i$. Then by the
hypothesis that $R$ has the compatibility property, if $x_{n+1}$ and
$x'_{n+1}$ lie in the same fibre they are equal. So $\pi$ is injective on $\Omega_R(S_1, \dots, S_n)$. 
Since $R$ is spectral,  $\Omega_R(S_1, \dots, S_n)$ is compact and open. Therefore, $\Omega_R$
is a well-defined operation on~$\cA^{\mathrm {co}}_\pi$.

Finally, we define conditions that morphisms of Hausdorff \'etale spaces are
required to satisfy, when those Hausdorff \'etale spaces are equipped with
additional relations.

\begin{definition}
Take a partial function $\varphi \from X \rightharpoonup Y$ and
$(n+1)$-ary relations $R_X$ and $R_Y$ on $X$ and $Y$. Then $\varphi$
satisfies the \defn{reverse forth condition} if whenever
$R_Xx_1{\dots} x_{n+1}$ and $\varphi(x_1), \dots, \varphi(x_n)$ are
defined, then $\varphi(x_{n+1})$ is defined and
$R_Y\varphi(x_1){\dots} \varphi(x_{n+1})$. The partial map $\varphi$
satisfies the \defn{back condition} if whenever $\varphi(x_{n+1})$
is defined and $R_Yy_1{\dots} y_n\varphi(x_{n+1})$, then there exist
$x_1, \dots, x_n \in \dom(\varphi)$ such that $\varphi(x_1) = y_1,
\dots, \varphi(x_n) = y_n$ and $R_Xx_1{\dots} x_{n+1}$.
\end{definition}

We are now ready to extend the adjunction $F \from \ralg \dashv \hauset^{\operatorname{op}} \,\from\! G$ to algebras with operators. We fix a functional signature $\sigma$ (disjoint from $\{-, \rest\}$). We have already defined $\ralg(\sigma)$.

\begin{definition}
The category $\hauset(\sigma)$ has
\begin{itemize}
\item objects: the objects of $\hauset$ equipped with, for each $\Omega
\in \sigma$, an $(n+1)$-ary tight spectral relation $R_\Omega$ that has the
compatibility property, where $n$ is the arity of~$\Omega$;
\item
morphisms: morphisms of $\hauset$ that satisfy the reverse forth condition and the back condition with respect to $R_\Omega$, for every $\Omega \in \sigma$.
\end{itemize}
\end{definition}

We are required to note at this point that both the reverse forth
condition and the back condition are preserved by composition of
partial maps (and are also satisfied by identity maps); hence
$\hauset(\sigma)$ is indeed a category.

\begin{theorem}\label{thm:expansion}
There is an adjunction $F' \from \ralg(\sigma) \dashv \hauset(\sigma)^{\operatorname{op}} \,\from\! G'$ that extends the adjunction $F \dashv G$ of \Cref{sec:duality} in the sense that the appropriate reducts of $F'(\algebra A)$ and $G'(\pi \from X \twoheadrightarrow X_0)$ equal $F(\algebra A)$ and $G(\pi \from X \twoheadrightarrow X_0)$, respectively.
\end{theorem}

The proof of this theorem takes up most of the remainder of the
paper. The reader is likely to have understood by now how to form $F'$
and $G'$. For $F'$: given a difference--restriction algebra $\algebra A$ equipped with compatibility preserving operators indexed by $\sigma$, take $F(\algebra
A)$ and equip it with, for each operator $\Omega$ on $\algebra A$,
the relation $R_\Omega$ defined according to
\eqref{make_relation}. For $G'$: given a Hausdorff \'etale space $\pi \from X
\twoheadrightarrow X_0$ equipped with relations, take $G(\pi \from X
\twoheadrightarrow X_0)$ and equip it with, for each relation $R$, the
operation $\Omega_R$ defined according to \eqref{make_operation}. The
proof consists therefore of establishing the following facts.

\begin{enumerate}
\item The $F'$ we wish to define is well-defined on objects. That is,
the defined relations $R_\Omega$ have the compatibility property (\Cref{l:10}) and are tight and spectral
(\Cref{l:10.2}).

\item
$F'$ is well-defined on morphisms. That is, for a morphism $h$ in
$\ralg$ that preserves additional operations, the partial map $Fh$
satisfies the reverse forth and the back conditions with respect to
each pair of relations (\Cref{l:11}).

\item
$G'$ is well-defined on objects: each defined $\Omega_R$ is a
compatibility preserving operator (\Cref{l:12}).

\item
$G'$ is well-defined on morphisms: given a morphism $H$ in $\hauset$,
the defined operations are preserved by $G\varphi$ (\Cref{l:13}).

\item The unit and counit used in \Cref{t:adj} are still permitted
families of morphisms. That is, for each algebra $\algebra A$, the
map $\eta_{\algebra A}$ preserves the additional operators, and for
each Hausdorff \'etale spaces $\pi$, the map $\lambda_\pi$ satisfies the reverse
forth condition and the back condition (\Cref{l:14}).
\end{enumerate}

\begin{lemma}\label{l:10}
If an operation $\Omega$ is compatibility preserving, then $R_\Omega$ has the compatibility property.
\end{lemma}

\begin{proof}
Take $\mu_1 C\nu_1, \dots ,\mu_nC\nu_n$ and $R_\Omega \mu_1\dots \mu_{n+1}$ and
$R_\Omega \nu_1\dots \nu_{n+1}$. For each $1 \le i \le n$, if $\mu_i \not\approx \nu_i$ then we can find $a_i \in \mu_i$ and $b_i \in \nu_i$ such that these are compatible elements of $\algebra A$: indeed, if there exist $a \in \mu$ and $b \in \nu$ such that $a \rest b \not\in \nu$ then we have $b - (a \rest b) \in \nu$ and we can set $a_i \coloneqq a$ and $b_i \coloneqq b - (a \rest b)$. If, on the other hand, $\mu_i \approx \nu_i$, then $\mu_i = \nu_i$ and we arbitrarily choose some element of this maximal filter to be both $a_i$ and $b_i$, so again these two elements of $\algebra A$ are compatible. 

Now since $\Omega$ is compatibility preserving, $\Omega(a_1, \dots, a_n)$ and $\Omega(b_1,
\dots, b_n)$ are compatible. We have, by the hypotheses and the
definition of $R_\Omega $, that $\Omega(a_1, \dots, a_n) \in 
\mu_{n+1}$ and $\Omega(b_1, \dots, b_n) \in \nu_{n+1}$. To see that $\mu_{n+1}C\nu_{n+1}$, assume  $\mu_{n+1}\approx\nu_{n+1}$. Then both $\mu_{n+1}$ and $\nu_{n+1}$ contain the element $c \coloneqq \Omega(a_1, \dots, a_n) \rest \Omega(b_1, \dots, b_n) = \Omega(b_1, \dots, b_n) \rest \Omega(a_1, \dots, a_n)$. To see that $\mu_{n+1} \subseteq \nu_{n+1}$, let $a \in \mu_{n+1}$. Then $a \cdot c \in \mu_{n+1}$. By $\mu_{n+1}\approx\nu_{n+1}$, we obtain $(a \cdot c) \rest c \in \nu_{n+1}$. But $(a \cdot c) \rest c = a \cdot c$, so $a \cdot c \in \nu_{n+1}$; hence $a \in \nu_{n+1}$. We have established the inclusion $\mu_{n+1} \subseteq \nu_{n+1}$, and hence by maximality the equality $\mu_{n+1} = \nu_{n+1}$, as required. 
\end{proof}

\begin{lemma}\label{l:10.2}
If an operation $\Omega$ is an operator, then $R_\Omega$ is tight and spectral.
\end{lemma}

\begin{proof}
Since sets of the form $\widehat a$ form a basis for the topology, in order to prove that $R_\Omega$ is continuous, it is sufficient to prove that for each $a_1,\dots,a_n \in \algebra A$, the set $\Omega_{R_\Omega}(\widehat a_1, \dots, \widehat a_n)$ is open. This relies on the fact that $\Omega_{R_\Omega}$ commutes with unions in any of its argument, which clearly holds any operation of the form $\Omega_{R}$. Further, since  sets of the form $\widehat a$ are \emph{compact} opens, any compact open is a finite union of sets of the form $\widehat a$. So in order to prove that if $S_1,\dots, S_n \subseteq X$ are compact opens, then $\Omega_{R_\Omega}(S_1, \dots, S_n)$ is a compact open, it is  sufficient to prove that for each $a_1,\dots,a_n \in \algebra A$, the set $\Omega_{R_\Omega}(\widehat a_1, \dots, \widehat a_n)$ is a compact open. This relies again on the fact that $\Omega_{R_\Omega}$ commutes with unions in any of its components as well as the fact that a finite union of compact sets is compact. In summary, to prove that $R_\Omega$ is spectral, it suffices to check that  each  $\Omega_{R_\Omega}(\widehat a_1, \dots, \widehat a_n)$ is a compact open.

We claim that $\Omega_{R_\Omega}(\widehat a_1, \allowbreak\dots, \widehat \allowbreak a_n) = \widehat {\Omega(a_1,\dots,a_n)}$. First note that if $\mu \in \Omega_{R_\Omega}(\widehat a_1, \dots, \widehat a_n)$, then there exist $\mu_1,\dots,\mu_n$ containing $a_1, \dots, a_n$, respectively, such that, in particular, $\Omega(a_1,\allowbreak \dots,\allowbreak a_n) \in \mu$. So $\mu \in \widehat {\Omega(a_1,\dots,a_n)}$. This gives the inclusion $\Omega_{R_\Omega}(\widehat a_1, \dots, \widehat a_n) \subseteq \widehat {\Omega(a_1,\dots,a_n)}$. 

Now suppose the inclusion is strict, so there exists $\nu$ in $\widehat {\Omega(a_1,\dots,a_n)}$ but not in  $\Omega_{R_\Omega}(\widehat a_1, \dots,\allowbreak \widehat a_n)$. Then for each $\mu_1 \in \widehat a_1,\dots,\mu_n \in \widehat a_n$, there are $a_{\mu_1} \in \mu_1, \dots, a_{\mu_n} \in \mu_n$ such that $\Omega( a_{\mu_1} , \dots,\allowbreak a_{\mu_n}) \not\in \nu $, and without loss of generality, $a_{\mu_i} \le a_i$, for each $i$ (which we assume without mention for elements of the $\mu_i$'s from now on). Then by compactness of  $\widehat a$, for each $\mu_1 \in \widehat a_1,\dots,\mu_{n-1} \in \widehat a_{n-1}$, there are $ a_{\mu_1} \in \mu_1, \dots, a_{\mu_{n-1}} \in \mu_{n-1}$ and finitely many $a_n^1, \dots a_n^k$ such that  $\Omega( a_{\mu_1} , \dots, a_{\mu_{n-1}}, a_n^j) \not\in \nu $ for each $j$, and $\join_j a_n^j = a_n$. Since $\Omega$ is an operator, this implies $\Omega(a_{\mu_1} , \dots,a_{\mu_{n-1}}, a_n) \not\in \nu $. Continuing this inductive argument, we arrive at $\Omega(a_1 , \dots, a_n) \not\in \nu $, contradicting $\nu$ in $\widehat {\Omega(a_1,\dots,a_n)}$. Hence the inclusion $\Omega_{R_\Omega}(\widehat a_1, \dots, \widehat a_n) \subseteq \widehat {\Omega(a_1,\dots,a_n)}$ is an equality, so  $\Omega_{R_\Omega}(\widehat a_1, \dots, \widehat a_n)$ is a compact open, as required. 

Finally, we prove that $R_\Omega$ is tight. If condition \eqref{eq:tight} holds, then in particular it holds when each $S_i$ is of the form $\widehat a_i$. So for maximal filters $\mu_1, \dots, \mu_n$, condition \eqref{eq:tight} implies that for all $a_1 \in \mu_1, \dots, a_n \in \mu_n$ we have $\mu_{n+1} \in \Omega_{R_\Omega}(\widehat a_1, \dots, \widehat a_n)$. But we just saw that $\Omega_{R_\Omega}(\widehat a_1, \allowbreak\dots, \widehat \allowbreak a_n) = \widehat {\Omega(a_1,\dots,a_n)}$, so $\mu_{n+1} \in \Omega_{R_\Omega}(\widehat a_1, \dots, \widehat a_n)$ is equivalent to $\Omega(a_1,\dots,a_n) \in \mu_{n+1}$. Thus we have $\{\Omega(a_1, \dots, a_n) \mid a_i \in \mu_i, 1\leq i \leq n\}\subseteq \mu_{n+1}$, which by the definition of $R_\Omega$ gives $R_\Omega \mu_1{\dots}\mu_{n+1}$. We conclude that $R_\Omega$ is tight.
\end{proof}

\begin{lemma}\label{l:11}

Let $h \from \algebra A \to \algebra B$ be a homomorphism
of difference--restriction algebras, and let
$\Omega^{\algebra A}$ and $\Omega^{\algebra B}$ be $n$-ary operators on $\algebra A$
and $\algebra B$ respectively. If $h$ validates
\[h(\Omega^{\algebra A}(a_1, \dots, a_n)) = \Omega^{\algebra
B}(h(a_1), \dots, h(a_n)),\]
then $Fh$ satisfies the reverse forth and the back conditions with
respect to $R_{\Omega^{\algebra B}}$ and $R_{\Omega^{\algebra A}}$.
\end{lemma}

\begin{proof}
We write $R_\cA$ for $R_{\Omega^{\algebra A}}$, and we write $R_\cB$
for $R_{\Omega^{\algebra B}}$. For the reverse forth condition,
suppose $R_\cB \nu_1\dots \nu_{n+1}$ holds and that $Fh(\nu_1), \dots,
Fh(\nu_n)$ are defined. Denote $Fh(\nu_1), \dots, \allowbreak Fh(\nu_n)$ by $\mu_1,
\dots, \mu_n$ respectively. By the definition of $Fh$, we have $\mu_i
=h^{-1} (\nu_i)$ for each $i$. Then for each $a_1 \in \mu_1, \dots, a_n \in \mu_n$ we have \begin{equation}\label{eq:homomorphism}
h(\Omega^{\algebra A}(a_1, \dots,
a_n)) = \Omega^{\algebra B}(h(a_1), \dots, h(a_n)) \in
\nu_{n+1},\end{equation} by the definition of $R_\cB$. Since $\mu_1 \times \dots \times \mu_n$ is nonempty, \eqref{eq:homomorphism} shows in particular that $h^{-1}(\nu_{n+1})$ is nonempty, and thus $Fh$ is defined at $\nu_{n+1}$, with  $Fh(\nu_{n+1}) = h^{-1}(\nu_{n+1})$. Then by the definition of $R_\cA$, \eqref{eq:homomorphism} further shows that  the relation
$R_\cA \mu_1{\dots} \mu_{n}Fh(\nu_{n+1})$ holds, and the reverse forth
condition is established.

For the back condition, suppose that $Fh(\nu_{n+1})$ is defined and
that the relation \[R_\cA \mu_1{\dots} \mu_nFh(\nu_{n+1})\] holds. 
Write $\mu_{n+1}$ for $Fh(\nu_{n+1})$. 
Fix some $a_1 \in \mu_1, \dots, a_n \in \mu_n$. It suffices to find ultrafilters $U_1, \dots, U_n$ of $h(a_1)^\downarrow, \dots, h(a_n)^\downarrow$, respectively, such that
\begin{enumerate}
\item \label{item:h}
for each $i$, if $a \le a_i$ then $h(a) \in U_i$;
\item \label{item:closed}
for each $b_1 \in U_1, \dots, b_n \in U_n$ we have $\Omega^\cB(b_1, \dots, b_n) \in \mu_{n+1}$.
\end{enumerate}
Then we can simply define $\nu_i = U_i^\uparrow$, for each $i$ (where the upwards closure is taken in $\cB$). By \cite[Proposition~5.4]{diff-rest1} these are elements of $\pf(\cB)$, and it is straightforward to check that $Fh(\nu_i) = \mu_i$ and $R_\cB\nu_1{\dots}\nu_{n+1}$, as required.

To obtain the $U_i$'s, we maintain a list $F_1 \subseteq h(a_1)^\downarrow, \dots F_n \subseteq h(a_n)^\downarrow$ such that
\begin{itemize}
\item 
each $F_i$ is nonempty;
\item 
each $F_i$ does not contain $0$;

\item 
each $F_i$ is closed under meets;

\item 
for each $b_1 \in F_1, \dots, b_n \in F_n$ we have $\Omega^\cB(b_1, \dots, b_n) \in \nu_{n+1}$.
\end{itemize}
We begin by setting $F_i \coloneqq \{h(a) \mid a \le a_i\}$ and checking that this list of $F_i$s has the required properties. Then we successively expand  each $F_i$, one-by-one, to an ultrafilter $U_i$, updating the $i$th element of our list to $U_i$ and checking that the properties are maintained by each update. 

For notational convenience, we explain how to perform the first update (of $F_1$); the other updates are analogous.
Define the set \[G \coloneqq \{ b \in h(a_1)^\downarrow \mid \exists b_2 \in F_2, \dots, b_n \in F_n : \Omega^\cB(\overline b, b_2, \dots, b_n) \not\in \nu_{n+1}\},\] where the complement $\overline b$ is taken in $h(a_1)^\downarrow$. 
We claim that $G$ is closed under meets. 
Indeed, suppose $b, b' \in h(a_1)^\downarrow$ and both $\Omega^\cB(\overline b, b_2, \dots, b_n) \not\in \nu_{n+1}$ and $\Omega^\cB(\overline {b'}, b_2, \dots, b_n) \not\in \nu_{n+1}$. For $2 \leq i \leq n$, define $c_i \coloneqq b_i \cdot b'_i$,  so $c_i \in F_i$. Since $\Omega^\cB$ is order preserving in each coordinate, we have $\Omega^\cB( \overline b,c_2,\dots,c_n) \not\in \nu_{n+1}$ and $\Omega^\cB( \overline {b'},c_2,\dots,c_{n}) \not\in \nu_{n+1}$.
Then as $h(a_1) -(b\cdot b') = (h(a_1) -b) + (h(a_1) - b')$, by additivity of $\Omega^\cB$, we have \[\Omega^\cB( \overline {b\cdot b'}, c_2, \dots,c_n)  = \Omega^\cB( \overline b,c_2,\dots,c_n)  + \Omega^\cB( \overline {b'},c_2,\dots,c_n) \not\in \nu_{n+1} .\footnote{The primality property of $\nu_{n+1}$ we are using here follows from \cite[Proposition~5.4]{diff-rest1}.}\]
Now the set $H \coloneqq \{b_1 \cdot b \mid b_1 \in F_1\text{ and }b \in G\}$ is nonempty and closed under meets. We check that it also does not contain $0$. So suppose for a contradiction that $b_1 \in F_1$, $b \in G$ and $b_1 \cdot b = 0$. Then $b_1 \le \overline b$. Since $b \in G$, there exist $b_2\in F_2,\dots,b_n\in F_n$ such that $\Omega^\cB(\overline b, b_2, \dots, b_n) \not\in \nu_{n+1}$. As $\Omega^\cB$ is order preserving, we obtain $\Omega^\cB(b_1, b_2, \dots, b_n) \not\in \nu_{n+1}$, contradicting our assumption about the current list of $F_i$s. 

Finally, extend $H$ to an ultrafilter $U_1$ of $h(a_1)^\downarrow$. Then $U_1$ is nonempty, does not contain $0$, and is closed under meets. So to be allowed to replace $F_1$ in our list with $U_1$ we only need to check that for each $b_1 \in U_1$ and $b_2 \in F_2, \dots, b_n \in F_n$ we have $\Omega^\cB(b_1, \dots, b_n) \in \nu_{n+1}$. But suppose $\Omega^\cB(b_1, \dots, b_n) \not\in \nu_{n+1}$ for such $b_1, \dots, b_n$. Then $\overline b_1 \in G$, we have $\overline b_1 \in U_1$. As $U_1$ is an ultrafilter, this implies $b_1 \not\in U_1$---a contradiction.

It is clear that the list of ultrafilters we obtain after performing all $n$ updates has the properties \eqref{item:h} and \eqref{item:closed}, as required.
\end{proof}

\begin{lemma}\label{l:12}
If a spectral relation $R$ has the compatibility property, then the operation $\Omega_R$ is a
compatibility preserving operator.
\end{lemma}
\begin{proof}
To see that $\Omega_R$ is compatibility preserving, first note that from the definition of $\rest_\pi$ \eqref{eq:dual_operations}, it is clear that any two $S, S' \in \cA^{\mathrm {co}}_\pi$ are compatible if and only if for each $x \in S$ and $x' \in S'$, we have $x C x'$. 
Let $S_i, S'_i
\in \cA^{\mathrm {co}}_\pi$ be compatible, for each $i$. 
Let $x_{n+1} \in \Omega_R(S_1 ,
\dots, S_n )$ and $x'_{n+1} \in \Omega_R( S'_1, \dots, S'_n)$. Then by the definition of $\Omega_R$, there exist $x_1 \in S_1, \dots, x_n \in S_n$ and $x'_1 \in S'_1, \dots , x'_n \in S'_{n+1}$ such that $Rx_1{\dots}x_{n+1}$ and $Rx'_1{\dots}x'_{n+1}$. For each $i$, since $S_i$ and $S'_i$ are compatible, we have $x_i C x'_i$. As $R$ has the compatibility property, we conclude that $x_{n+1} C x'_{n+1}$. Since $x_{n+1}$ and $x'_{n+1}$ were arbitrary elements of $\Omega_R(S_1 ,
\dots, S_n )$ and $ \Omega_R( S'_1, \dots, S'_n)$, respectively, we conclude that $\Omega_R(S_1 ,
\dots, S_n )$ and $ \Omega_R( S'_1, \dots, S'_n)$ are compatible.

To see that $\Omega_R$ is normal, simply note that if $S_i = \emptyset$, then the condition $\exists x_i \in S_i$ in \eqref{make_operation} cannot be witnessed.  To see that~$\Omega_R$ is additive, recall that binary joins in $\cA^{\mathrm {co}}_\pi$ are given by binary unions, and that from its definition, $\Omega_R$ clearly commutes with unions in any of its components. 
\end{proof}

\begin{lemma}\label{l:13}
Let $\varphi \from X \parrow Y$ define a morphism in $\hauset$ from
$(\pi \from X \twoheadrightarrow X_0)$ to $(\rho \from Y
\twoheadrightarrow Y_0)$, and let $R_X$ and $R_Y$ be $(n+1)$-ary
spectral relations on $X$ and $Y$ respectively, both having the compatibility
property. If $\varphi$ satisfies the reverse forth and the back
conditions with respect to $R_X$ and $R_Y$, then the $n$-ary
operations $\Omega_{R_X}$ and $\Omega_{R_Y}$ validate
\begin{equation}\label{eq:morphism}
G\varphi(\Omega_{R_Y}(Y_1, \dots, Y_n)) =
\Omega_{R_X}(G\varphi(Y_1), \dots, G\varphi(Y_n)).\end{equation}
\end{lemma}

\begin{proof}
Recall that $G\varphi$ is given by $\varphi^{-1}$
(inverse image). The proof is a verification that for $Y_1, \dots, Y_n \in G(\rho)$ we have 
\[\varphi^{-1}(\Omega_{R_Y}(Y_1, \dots, Y_n)) =
\Omega_{R_X}(\varphi^{-1}(Y_1), \dots, \varphi^{-1}(Y_n)),\]
and is identical to the proof in the discrete case, \cite[Lemma 5.10]{diff-rest2}. (The conditions that $R_X$ and $R_Y$ are spectral and have the compatibility property are only necessary to ensure that the terms in \eqref{eq:morphism} are defined.)
%
\end{proof}

\begin{lemma}\label{l:14}
Let $\algebra A$ be a difference--restriction algebra
and $\Omega$ be a compatibility preserving 
$n$-ary operator on $\algebra A$. Then the map $\eta_\algebra A$
used in \Cref{t:adj} validates
\[\eta_{\algebra A}(\Omega(a_1, \dots, a_n)) =
\Omega_{R_\Omega}(\eta_{\algebra A}(a_1), \dots, \eta_{\algebra
A}(a_n)).\]

Let $\pi \from X \twoheadrightarrow X_0$ be a Hausdorff \'etale space, and $R$ be
an $(n+1)$-ary tight spectral relation on $X$ that has the compatibility property. The map $\lambda_\pi$ used in \Cref{t:adj} satisfies the reverse forth
condition and the back condition with respect to $R$ and
$R_{\Omega_R}$.
\end{lemma}

\begin{proof}
The first part is simply $\widehat {\Omega(a_1,\dots,a_n)} = \Omega_{R_\Omega}(\widehat a_1, \dots, \widehat a_n)$, which we showed in the proof of \Cref{l:10.2}.

For the reverse forth condition, suppose that $Rx_1{\dots} x_{n+1}$ and that $\lambda_\pi(x_1), \dots, \lambda_\pi(x_n)$ are defined. By the definition of $\lambda_\pi$, this means there exist $U_1 \ni x_1, \dots, U_n \ni x_n$ that are compact and open and on which $\pi$ is injective. Then $x_{n+1} \in \Omega_R(U_1, \dots, U_n)$, and by the properties of $R$, we know that $\Omega_R(U_1, \dots, U_n)$ is compact and open and that $\pi$ is injective on this set. Thus $\lambda^0_\pi(x_{n+1})$ is nonempty, which implies $\lambda_\pi(x_{n+1})$ is defined. By the same reasoning, we see that for \emph{any} $U_1 \in \lambda_\pi(x_1), \dots, U_n \in \lambda_\pi(x_n)$ we have $\Omega_R(U_1, \dots, U_n) \in \lambda_\pi(x_{n+1})$, which by the definition of $R_{\Omega_R}$ gives $R_{\Omega_R}\lambda_\pi(x_1){\dots}\lambda_\pi(x_{n+1})$, as required.

For the back condition, suppose that $\lambda_\pi(x_{n+1})$
is defined and $R_{\Omega_R}\mu_1{\dots} \mu_n\lambda_\pi(x_{n+1})$.  Then for each $i$, the set $\bigcap \mu_i$ is nonempty (see the proof of \Cref{t:discrete-duality} for the argument). Pick some $x_1 \in \bigcap \mu_1, \dots, x_n \in \bigcap \mu_n$. Then for each $i$, we have $\lambda_\pi(x_i) \supseteq \mu_i$, and hence $\lambda_\pi(x_i) = \mu_i$. Further, if $S_1, \dots, S_n$ are compact opens containing $x_1, \dots, x_n$, respectively, then there exist $U_1 \subseteq S_1, \dots, U_n \subseteq S_n$ with $U_1 \in \mu_1, \dots, U_n \in \mu_n$. Then by the definition of $R_{\Omega_R}$ we know $x_{n+1} \in \Omega_R(U_1, \dots, U_n) \subseteq \Omega_R(S_1, \dots, S_n)$. As $R$ is tight, we conclude that $R x_1{\dots}x_{n+1}$, as required. 
\end{proof}

This completes the proof that $F' \dashv G'$, and hence the proof of \Cref{thm:expansion}.

It is now straightforward to extend the completeness and duality
results of the previous section.

\begin{definition}\label{def:completion'}
Let $\algebra A$ be an algebra in $\ralg(\sigma)$.
A \defn{finitary compatible completion} of $\algebra A$  is an embedding $\iota \from \algebra A
\hookrightarrow \algebra C$ of $(\{-, \rest\} \cup
\sigma)$-algebras such that $\cC$ is in $\ralg(\sigma)$ and finitarily compatibly complete, and
$\iota[\cA]$ is finite-join dense in~$\cC$.
\end{definition}

\begin{corollary}\label{p:1_0'}
Let $\algebra A$ be an algebra in $\ralg(\sigma)$.
If $\iota \from \algebra A \hookrightarrow \algebra C$ and $\iota'
\from \algebra A \hookrightarrow \algebra C'$ are finitary compatible completions of $\cA$ then there is a unique
isomorphism $\theta \from \algebra C \to \algebra C'$ satisfying the
condition $\theta \circ \iota = \iota'$.
\end{corollary}
\begin{proof}
Use the isomorphism $\theta$ from \Cref{p:1_0}. The fact that $\iota[\algebra A]$ is finite-join dense in $\algebra C$ and the normality and additivity of the additional operations ensure those additional operators are preserved by $\theta$.
\end{proof}

\begin{corollary}\label{p:completion'}
For every algebra $\cA$ in $\ralg(\sigma)$, the embedding
$\eta_\cA \from \cA \hookrightarrow (G' \circ F') (\cA)$ is the
finitary compatible completion of $\cA$.
\end{corollary}

Let $\fralg(\sigma)$ be the full subcategory of $\ralg(\sigma)$
consisting of the finitarily compatibly complete algebras and $\lczdet(\sigma)$ be the full subcategory of $\hauset(\sigma)$ consisting of the locally compact zero-dimensional spaces.

\begin{corollary}\label{c:extended-duality}
There is a duality between the categories $\fralg(\sigma)$ and
$\lczdet(\sigma)$.
\end{corollary}

\begin{proof}
Given \Cref{t:discrete-duality}, we only need to check that the
families of functions $\eta_{\algebra A}$ and $\lambda_\pi$ are
still isomorphisms in the expanded categories, for which it only
remains to show that the functions $\eta_{\algebra A}^{-1}$ and
$\lambda_\pi^{-1}$ are valid morphisms in the expanded categories.

We know $\eta_{\algebra A}$ is a bijection and preserves additional
operations, and it is an elementary algebraic fact that this implies
its inverse $\eta_{\algebra A}^{-1}$ preserves those same additional
operations. Thus $\eta_{\algebra A}^{-1}$ is a morphism.

For $\lambda_\pi^{-1}$, we must check that the reverse forth
condition and the back condition are satisfied with respect to
additional relations.  But we see from the proof of \Cref{l:14} that
$\lambda_\pi$ both preserves and
reflects each additional relation. Hence $\lambda_\pi^{-1}$ also preserves and reflects additional relations. As $\lambda_\pi^{-1}$ is a
bijection, it is then evident that the reverse forth and back
conditions are respected.
\end{proof}

\begin{corollary}
The category $\fralg(\sigma)$ is a reflective subcategory of $\ralg(\sigma)$.
\end{corollary}

\begin{corollary}
The category $\lczdet(\sigma)$ is a reflective subcategory of $\hauset(\sigma)$.
\end{corollary}

\section{Subtraction algebras}\label{sec:subtraction}

In this section we will see that the adjunction of \Cref{t:adj}
induces a dual adjunction between the category $\subalg$ of
subtraction algebras and $\{-\}$-homomorphisms and the category
$\hausp$ of Hausdorff spaces and continuous and proper partial
functions (\Cref{p:3}). Restricting to fixpoints yields the classical
Stone dualities for generalised Boolean algebras (\Cref{c:1} and
\Cref{c:2}).\footnote{The object part of these dualities is treated
  in~\cite{Stone1937}, while the full categorical duality of
  \Cref{c:2} is established in~\cite{Doctor1964}. See
  also~\cite{Lawson2023} for a detailed up-to-date treatment. Although
  the authors are not aware of the origins of the duality stated in
  \Cref{c:1}, it can be found, e.g., in the preliminary Section~3 of
  \cite{skewwithintersection}.}

We start by observing that the categories
$\subalg$ and $\hausp$ may indeed be seen as full subcategories of
$\ralg$ and $\hauset$, respectively. Indeed, given a subtraction
algebra $(\cA, -)$, it is easy to check that setting $\rest = \cdot$
yields a difference--restriction algebra structure on $\cA$. On the
other hand, $\hausp$ may be identified with the full subcategory of
$\hauset$ whose \'etale maps are an identity function. Note
that when that is the case, conditions \ref{item:Q1} -- \ref{item:Q3}
are trivially satisfied.

\begin{proposition}\label{p:3}
The adjunction of \Cref{t:adj} restricts and co-restricts to an
adjunction between the categories $\subalg$ and
$\hausp^{\operatorname{op}}$.
\end{proposition}
\begin{proof}
It suffices to show the following:
\begin{enumerate}[label = (\roman*)]
\item\label{item:1} if $\cA$ is a subtraction algebra then $F(\cA) \in \hausp$;
\item\label{item:5} if $X$ is a Hausdorff space and $\mathrm{id}_X$ the identity
map on $X$, then $G(\mathrm{id}_X)$ is a subtraction algebra.
\end{enumerate}
To prove \ref{item:1}, it suffices to show that for a subtraction
algebra $\cA$, the relation $\approx_{\pf\cA}$ is given by filter
equality.  Let $\mu, \nu \in \pf(\cA)$ be
$\approx_{\pf(\cA)}$-equivalent and fix an element $b \in \nu$.
Then for all $a \in \mu$, by~\eqref{eq:2}, we have that $a \rest b
= a \cdot b$ belongs to $\nu$. Since $\nu$ is upward closed, we also
have $a \in \nu$, which shows that $\mu \subseteq \nu$. By
maximality, it then follows that $\mu = \nu$, as required.

For \ref{item:5}, we just need to observe that $\ko{\mathrm{id}_X}$
consists of all compact open subsets of~$X$ and, by~\eqref{eq:dual_operations},
the operation $\rest_{\mathrm{id}_X}$ is given by set-theoretical
intersection, that is, coincides with the operation
$\cdot_{\mathrm{id}_X}$ determined by $-_{{\mathrm{id}_X}}$. Thus
$G({\mathrm{id}_X})$ is a subtraction algebra.
\end{proof}

Now, by \Cref{t:discrete-duality} and \Cref{r:3}, we know that the
adjunction of \Cref{p:3} induces a duality between the full
subcategories of $\subalg$ and of $\hausp$ determined, respectively,
by those subtraction algebras having all finite compatible joins and
by the generalised Stone spaces. Thus, in order to obtain the claimed
dualities, we shall characterise the former.

Recall that a \defn{generalised Boolean algebra} is a distributive
lattice with a bottom element, $\bot$, and a relative complement
operation, $-$, validating $a \wedge (b - a) = \bot$ and $a \vee (b -
a) = a \vee b$.
\begin{lemma}
Let $\cA$ be a subtraction algebra. Then $\cA$ has all finite
compatible joins if and only if $\cA$ admits a generalised Boolean
algebra structure whose partial order relation and relative
complement operation are those inherited from $\cA$.
\end{lemma}
\begin{proof}
Suppose that $\cA$ is a subtraction algebra having all finite
compatible joins. Since, in a subtraction algebra, every two
elements are compatible, $\cA$ has all binary joins, which shows
that $\cA$ is a lattice. It remains to show that for all $a,
b, c \in \cA$, the following equalities hold:
\begin{enumerate}[label = (\roman*)]
\item\label{item:6} $a \vee (b \cdot c) = (a \vee b) \cdot (a \vee
c)$ (that is, $\cA$ is distributive),
\item \label{item:7} $a \cdot (b-a) = 0$,
\item \label{item:8} $a \vee (b-a) = a \vee b$.
\end{enumerate}
Let $x = a \vee b \vee c$. By \cite[page
2154]{schein1992difference}, $x^{\downarrow}$ is a Boolean algebra
for the partial order inherited from $\cA$, with top element $x$,
bottom element $0$ and complement given by $\overline{z} = x-z$, for
all $z \in x^\downarrow$. In particular, $x^\downarrow$ is a
distributive lattice, which shows that~\ref{item:6}
holds. Equality~\ref{item:7} is trivially satisfied in every
subtraction algebra. Finally, to prove~\ref{item:8}, observing that
$b-a = b \cdot (x-a)$, we may compute
\[a \vee (b -a) = a \vee (b \cdot (x-a))
\just{\text{\ref{item:6}}}= (a \vee b) \cdot (a \vee (x-a)) =
(a\vee b) \cdot (a \vee \overline{a}) = a \vee b. \popQED\]
\end{proof}

Let $\ba^+$ denote the category of generalised Boolean algebras and
$\{-\}$-homomorphisms and $\stone^+$ denote the category of
generalised Stone spaces and continuous and proper partial functions.
Note that a map between generalised Boolean algebras is a
$\{-\}$-homo\-morphism if and only if it is a lattice
homomorphism. Indeed, if  $h\colon \cA_1 \to \cA_2$ is a map between
generalised Boolean algebras that preserves $\wedge$ and
$\vee$, then since $a - b$ is the only element $c$ satisfying $b\wedge c =
\bot$ and $b\vee c = a\vee b$, it follows that $h$ preserves $-$. The
converse is a consequence of \Cref{l:6}.

As an immediate consequence of the above results, we obtain the following.
\begin{corollary}\label{c:1}
There is a duality between $\ba^+$ and $\stone^+$.
\end{corollary}

Finally, it is natural to ask what is the duality obtained if, instead
of $\stone^+$, we considered its wide subcategory determined by the
continuous and proper (total) functions. The answer is given by the
following result, containing the most well-known version of Stone
duality for generalised Boolean algebras.  We recall
that a homomorphism $h: \cA \to \cB$ between generalised Boolean
algebras is called \defn{proper} if $\cB = \bigcup_{a \in \cA}
h(a)^\downarrow$.

\begin{corollary}\label{c:2}
There is a duality between the category of generalised Boolean algebras with proper
homomorphisms and the category of generalised Stone spaces with continuous
and proper functions.
\end{corollary}
\begin{proof}
Given a proper homomorphism $h: \cA \to \cB$, let us first show that
$Fh$ is a total map. Given $\mu \in \pf(\cB)$, let $b \in
\mu$. Since $h$ is proper, there exists some $a \in \cA$ satisfying
$b \leq h(a)$ and thus, $h(a) \in \mu$, that is, $\mu \in
\widehat{h(a)}$. In particular, $\mu \in \dom(Fh)$, as required.

Conversely, suppose that $\varphi: X \to Y$ is a proper continuous
function between generalised Stone spaces and let $U \subseteq X$ be
a compact open subset. For proving that $G\varphi$ is proper, we
need to show the existence of a compact open subset $V \subseteq Y$
satisfying $U \subseteq \varphi^{-1}(V)$ or, equivalently,
$\varphi(U) \subseteq V$. Since $\varphi$ is a total function, for
each $x \in X$, the value $\varphi(x)$ is defined, and since $Y$ is a
generalised Stone space, there exists a compact open subset $V_x
\subseteq Y$ such that $\varphi(x) \in V_x$. Clearly, we have
$\varphi(U) \subseteq \bigcup_{x \in X} V_x$. Now since continuous
images of compact subsets are compact, there exist $x_1, \dots, x_n
\in X$ such that $\varphi(U) \subseteq \bigcup_{i = 1}^n
V_{x_i}$. Thus taking $V = \bigcup_{i = 1}^n V_{x_i}$ yields the
required compact open subset of $Y$.
\end{proof}

\section*{Funding}
C\'elia Borlido was partially supported by the Centre for Mathematics
of the University of Coimbra - UIDB/00324/2020, funded by the
Portuguese Government through FCT/MCTES.
Ganna Kudryavsteva was supported by the ARIS grants P1-0288 and J1-60025. Brett McLean was supported by the FWO Senior Postdoctoral Fellowship 1280024N.

\bibliographystyle{amsplain}

\bibliography{brettbib}

\providecommand{\bysame}{\leavevmode\hbox to3em{\hrulefill}\thinspace}
\providecommand{\MR}{\relax\ifhmode\unskip\space\fi MR }
\providecommand{\MRhref}[2]{%
  \href{http://www.ams.org/mathscinet-getitem?mr=#1}{#2}
}
\providecommand{\href}[2]{#2}
\begin{thebibliography}{10}

\bibitem{completion}
Bernhard Banaschewski and G\"unter Bruns, \emph{Categorical characterization of the {M}ac{N}eille completion}, Archiv der Mathematik \textbf{18} (1967), 369--377.

\bibitem{skewwithintersection}
Andrej Bauer and Karin Cvetko-Vah, \emph{Stone duality for skew boolean algebras with intersections}, Houston Journal of Mathematics \textbf{39} (2013), no.~1, 73--109.

\bibitem{Bauer_2013}
Andrej Bauer, Karin Cvetko-Vah, Mai Gehrke, Samuel~J. van Gool, and Ganna Kudryavtseva, \emph{A non-commutative {P}riestley duality}, Topology and its Applications \textbf{160} (2013), no.~12, 1423--1438.

\bibitem{BERENDSEN2010141}
Jasper Berendsen, David~N. Jansen, Julien Schmaltz, and Frits~W. Vaandrager, \emph{The axiomatization of override and update}, Journal of Applied Logic \textbf{8} (2010), no.~1, 141--150.

\bibitem{blackburn_rijke_venema_2001}
Patrick Blackburn, Maarten~de Rijke, and Yde Venema, \emph{Modal logic}, Cambridge Tracts in Theoretical Computer Science, Cambridge University Press, Cambridge, 2001.

\bibitem{Boole1854}
George Boole, \emph{An investigation of the laws of thought on which are founded the mathematical theories of logic and probabilities}, Walton and Maberly, London, 1854.

\bibitem{diff-rest1}
C\'elia Borlido and Brett McLean, \emph{Difference--restriction algebras of partial functions: axiomatisations and representations}, Algebra Universalis \textbf{83} (2022), no.~3.

\bibitem{diff-rest2}
{C}\'elia Borlido and Brett McLean, \emph{Difference--restriction algebras of partial functions with operators: discrete duality and completion}, Journal of Algebra \textbf{604} (2022), 760--789.

\bibitem{COCKETT2021108030}
Robin Cockett and Richard Garner, \emph{Generalising the étale groupoid–complete pseudogroup correspondence}, Advances in Mathematics \textbf{392} (2021), 108030.

\bibitem{Doctor1964}
Hoshang~P. Doctor, \emph{The categories of {B}oolean lattices, {B}oolean rings and {B}oolean spaces}, Canad. Math. Bull. \textbf{7} (1964), 245--252.

\bibitem{1018.20057}
Wies{\l}aw~A. Dudek and Valentin~S. Trokhimenko, \emph{{Functional Menger $\mathcal P$-algebras}}, Communications in Algebra \textbf{30} (2002), no.~12, 5921--5931.

\bibitem{10.1145/2984450.2984453}
Emmanuel Filiot and Pierre-Alain Reynier, \emph{Transducers, logic and algebra for functions of finite words}, ACM SIGLOG News \textbf{3} (2016), no.~3, 4--19.

\bibitem{garvacii71}
Vladimir~S. Garvac'ki\u{\i}, \emph{$\cap$-semigroups of transformations}, Teoriya Polugrupp i ee Prilozheniya \textbf{2} (1971), 3--13 (Russian).

\bibitem{GEHRKE2001345}
Mai Gehrke and John Harding, \emph{Bounded lattice expansions}, Journal of Algebra \textbf{238} (2001), no.~1, 345--371.

\bibitem{goldblatt}
Robert Goldblatt, \emph{Metamathematics of modal logic, {P}art {I}}, Reports on Mathematical Logic \textbf{6} (1976), 41--77.

\bibitem{hirsch}
Robin Hirsch, Marcel Jackson, and Szabolcs Mikul{\'a}s, \emph{The algebra of functions with antidomain and range}, Journal of Pure and Applied Algebra \textbf{220} (2016), no.~6, 2214--2239.

\bibitem{disjoint}
Robin Hirsch and Brett McLean, \emph{Disjoint-union partial algebras}, Logical Methods in Computer Science \textbf{13} (2017), no.~2:10, 1--31.

\bibitem{1182.20058}
Ma{r}cel Jackson and Tim Stokes, \emph{{Partial maps with domain and range: extending Schein's representation}}, Communications in Algebra \textbf{37} (2009), no.~8, 2845--2870.

\bibitem{DBLP:journals/ijac/JacksonS11}
Mar{c}el Jackson and Tim Stokes, \emph{Modal restriction semigroups: towards an algebra of functions}, International Journal of Algebra and Computation \textbf{21} (2011), no.~7, 1053--1095.

\bibitem{JACKSON2015259}
Marcel Jackson and Tim Stokes, \emph{Monoids with tests and the algebra of possibly non-halting programs}, Journal of Logical and Algebraic Methods in Programming \textbf{84} (2015), no.~2, 259--275.

\bibitem{JACKSON2021106532}
Marcel Jackson and Tim {S}tokes, \emph{Override and update}, Journal of Pure and Applied Algebra \textbf{225} (2021), no.~3, 106532.

\bibitem{jackson2021restriction}
Marcel Jackson and Tim Stokes, \emph{{Restriction in Program Algebra}}, Logic Journal of the IGPL \textbf{31} (2022), no.~5, 926--960.

\bibitem{1951}
Bjarni J\'onsson and Alfred Tarski, \emph{Boolean algebras with operators. {P}art {I}}, American Journal of Mathematics \textbf{73} (1951), no.~4, pp. 891--939.

\bibitem{d9b8bf93-569b-32ea-ac64-293dc8ed9a8b}
Bjarni Jónnson and Alfred Tarski, \emph{Boolean algebras with operators. {P}art {II}}, American Journal of Mathematics \textbf{74} (1952), no.~1, 127--162.

\bibitem{Kudryavtseva2012}
Ganna Kudryavtseva, \emph{A refinement of {S}tone duality to skew {B}oolean algebras}, Algebra Universalis \textbf{67} (2012), no.~4, 397--416.

\bibitem{KUDRYAVTSEVA2025110313}
Ganna Kudryavtseva, \emph{Relating ample and biample topological categories with {B}oolean restriction and range semigroups}, Advances in Mathematics \textbf{474} (2025), 110313.

\bibitem{kudryavtseva2016boolean}
Ganna Kudryavtseva and {M}ark~V. Lawson, \emph{Boolean sets, skew {B}oolean algebras and a non-commutative {S}tone duality}, Algebra Universalis \textbf{75} (2016), no.~1, 1--19.

\bibitem{kudryavtseva2017perspective}
{G}anna Kudryavtseva and Mark~V. Lawson, \emph{A perspective on non-commutative frame theory}, Advances in Mathematics \textbf{311} (2017), 378--468.

\bibitem{lawson_2010}
Mark~V. Lawson, \emph{A noncommutative generalization of {S}tone duality}, Journal of the Australian Mathematical Society \textbf{88} (2010), no.~3, 385--404.

\bibitem{lawson2012non}
{M}ark~V. Lawson, \emph{Non-commutative {S}tone duality: inverse semigroups, topological groupoids and {C}*-algebras}, International Journal of Algebra and Computation \textbf{22} (2012), no.~06, 1250058.

\bibitem{LAWSON201677}
Mark~V. Lawson, \emph{Subgroups of the group of homeomorphisms of the {C}antor space and a duality between a class of inverse monoids and a class of {H}ausdorff \'etale groupoids}, Journal of Algebra \textbf{462} (2016), 77--114.

\bibitem{Lawson2023}
Mark~V. Lawson, \emph{Non-commutative {S}tone duality}, Semigroups, algebras and operator theory, Springer Proc. Math. Stat., vol. 436, Springer, Singapore, 2023, pp.~11--66.

\bibitem{LAWSON2013117}
Mark~{V}. Lawson and Daniel~H. Lenz, \emph{Pseudogroups and their \'etale groupoids}, Advances in Mathematics \textbf{244} (2013), 117--170.

\bibitem{Leech1990}
Jonathan Leech, \emph{Skew {B}oolean algebras}, Algebra Universalis \textbf{27} (1990), no.~4, 497--506.

\bibitem{Leech19967}
Jonathan {L}eech, \emph{Recent developments in the theory of skew lattices}, Semigroup Forum \textbf{52} (1996), no.~1, 7--24.

\bibitem{phdthesis}
Brett Mc{L}ean, \emph{Algebras of partial functions}, Ph.D. thesis, University College London, 2018.

\bibitem{2009.07895}
Brett {M}cLean, \emph{A categorical duality for algebras of partial functions}, Journal of Pure and Applied Algebra \textbf{225} (2021), no.~11, 106755.

\bibitem{finiterep}
Brett McLean and Szabolcs Mikul\'as, \emph{The finite representation property for composition, intersection, domain and range}, International Journal of Algebra and Computation \textbf{26} (2016), no.~6, 1199--1216.

\bibitem{Schein1970}
Boris~M. Schein, \emph{Relation algebras and function semigroups}, Semigroup Forum \textbf{1} (1970), no.~1, 1--62.

\bibitem{schein}
Boris~M. Schein, \emph{Restrictively multiplicative algebras of transformations}, Izvestija Vys\v{s}ih U\v{c}ebnyh Zavedeni\u{\i} Matematika \textbf{95} (1970), no.~4, 91--102 (Russian).

\bibitem{schein1992difference}
Boris~M. Schein, \emph{Difference semigroups}, Communications in Algebra \textbf{20} (1992), no.~8, 2153--2169.

\bibitem{Stokes2010}
Tim Stokes, \emph{Comparison semigroups and algebras of transformations}, Semigroup Forum \textbf{81} (2010), no.~2, 325--334.

\bibitem{Stone1937}
Marshall~H. Stone, \emph{Applications of the theory of {B}oolean rings to general topology}, Trans. Amer. Math. Soc. \textbf{41} (1937), no.~3, 375--481.

\bibitem{Tro73}
Valentin~S. Trokhimenko, \emph{Menger's function systems}, Izvestija Vys\v{s}ih U\v{c}ebnyh Zavedeni\u{\i} Matematika (1973), 71--78 (Russian).

\bibitem{wagnergeneralised}
Viktor~V. Wagner, \emph{Generalised groups}, Proceedings of the USSR Academy of Sciences \textbf{84} (1952), 1119--1122 (Russian).

\bibitem{Vagner1960}
Viktor~V. Wagner, \emph{Transformative semigroups}, Izvestija Vys\v{s}ih U\v{c}ebnyh Zavedeni\u{\i} Matematika \textbf{1960} (1960), no.~4 (17), 36--48 (Russian).

\bibitem{vagner1962}
{V}iktor~V. Wagner, \emph{Restrictive semigroups}, Izvestija Vys\v{s}ih U\v{c}ebnyh Zavedeni\u{\i} Matematika \textbf{1962} (1962), no.~6 (31), 19--27 (Russian).

\end{thebibliography}


\end{document}